\documentclass[reqno,10pt,centertags,draft]{amsart}
\usepackage{amsmath,amsthm,amscd,amssymb,latexsym,upref,enumerate}

\newtheorem{theorem}{Theorem}[section]
\newtheorem{lemma}[theorem]{Lemma}

\newtheorem{hypothesis}[theorem]{Hypothesis}
\theoremstyle{remark}
\newtheorem{remark}[theorem]{Remark}

\theoremstyle{definition}
\newtheorem{definition}[theorem]{Definition}
\numberwithin{equation}{section}

\newcommand{\bbA}{{\mathbb{A}}}

\newcommand{\bbC}{{\mathbb{C}}}

\newcommand{\bbN}{{\mathbb{N}}}
\newcommand{\bbR}{{\mathbb{R}}}

\newcommand{\bbZ}{{\mathbb{Z}}}

\newcommand{\dott}{\,\cdot\,}
\newcommand{\no}{\notag}
\newcommand{\lb}{\label}
\newcommand{\f}{\frac}

\newcommand{\ol}{\overline}

\newcommand{\wti}{\widetilde}

\newcommand{\bs}{\backslash}

\newcommand{\rank}{\text{\rm{rank}}}

\newcommand{\Arc}{\text{\rm{Arc}}}

\newcommand{\bi}{\bibitem}
\newcommand{\hatt}{\widehat}

\newcommand{\st}{\;|\;}

\newcommand{\deven}{\delta_{\rm even}}
\newcommand{\dodd}{\delta_{\rm odd}}

\renewcommand{\Re}{\text{\rm Re}}
\renewcommand{\Im}{\text{\rm Im}}

\renewcommand{\ge}{\geqslant}

\DeclareMathOperator{\diag}{diag}

\newcommand{\abs}[1]{\left\lvert#1\right\rvert}
\newcommand{\norm}[1]{\left\Vert#1\right\Vert}

\newcommand{\Om}{\Omega}

\newcommand{\si}{\sigma}

\newcommand{\al}{\alpha}
\newcommand{\be}{\beta}
\newcommand{\ga}{\gamma}
\newcommand{\De}{\Delta}
\newcommand{\de}{\delta}
\newcommand{\te}{\theta}
\newcommand{\Te}{\Theta}
\newcommand{\ta}{\tau}
\newcommand{\ze}{\zeta}
\newcommand{\ka}{\kappa}
\newcommand{\ea}{\eta}

\newcommand{\C}{\mathbb{C}}
\newcommand{\Cz}{{\C\backslash\{0\}}}
\newcommand{\CdD}{{\C\backslash\dD}}

\newcommand{\D}{\mathbb{D}}
\newcommand{\dD}{{\partial\hspace*{.2mm}\mathbb{D}}}
\newcommand{\Z}{{\mathbb{Z}}}
\newcommand{\N}{{\mathbb{N}}}

\newcommand{\M}{{\mathcal{M}}}
\newcommand{\A}{\mathbb{A}}
\newcommand{\U}{{\mathbb{U}}}
\newcommand{\V}{{\mathbb{V}}}
\newcommand{\W}{{\mathbb{W}}}
\newcommand{\T}{{\mathbb{T}}}

\newcommand{\Cm}{{\mathbb{C}^{m\times m}}}

\newcommand{\s}[1]{{\mathrm{s}(#1)}}
\newcommand{\sm}[1]{{\mathrm{s}(#1)^m}}
\newcommand{\smn}[1]{{\mathrm{s}(#1)^{m\times n}}}
\newcommand{\smm}[1]{{\mathrm{s}(#1)^{m\times m}}}
\newcommand{\lt}[1]{{\ell^2(#1)}}
\newcommand{\ltm}[1]{{\ell^2(#1)^m}}

\newcommand{\ltmn}[1]{{\ell^2(#1)^{m\times n}}}
\newcommand{\ltmm}[1]{{\ell^2(#1)^{m\times m}}}

\begin{document}
\title[Minimal Rank Decoupling of CMV Operators]{Minimal Rank Decoupling of Full-Lattice CMV Operators with
Scalar- and Matrix-Valued\\ Verblunsky Coefficients}

\author[S.\ Clark, F.\ Gesztesy, and M.\ Zinchenko]
{Stephen Clark, Fritz Gesztesy, and Maxim Zinchenko}

\address{Department of Mathematics \& Statistics,
University of Missouri, Rolla, MO 65409, USA}
\email{sclark@umr.edu}
\urladdr{http://web.umr.edu/\~{}sclark/index.html}

\address{Department of Mathematics,
University of Missouri, Columbia, MO 65211, USA}
\email{fritz@math.missouri.edu}
\urladdr{http://www.math.missouri.edu/personnel/faculty/gesztesyf.html}

\address{Department of Mathematics,
California Institute of Technology, Pasadena, CA 91125, USA}
\email{maxim@caltech.edu}
\urladdr{http://www.math.caltech.edu/\~{}maxim}

%

\thanks{Appeared in {\it Difference Equations and Applications}, Proceeding of 
the 14th International Conference on Difference Equations and Applications, Istanbul, July 21--25, 2008, M.\ Bohner, Z.\ Do{\v s}l\'a, G.\ Ladas, M.\ \"Unal, and 
A.\ Zafer (eds.), U{\u g}ur--Bah{\c c}e{\c s}ehir University Publishing Company, 
Istanbul, Turkey, 2009, pp.\ 19--59.}


\maketitle

\begin{abstract}
Relations between half- and full-lattice CMV operators with scalar- and
matrix-valued Verblunsky coefficients are investigated. In particular, the
decoupling of full-lattice CMV operators into a direct sum of two half-lattice CMV
operators by a perturbation of minimal rank is studied. Contrary to the Jacobi case, decoupling a full-lattice CMV matrix by changing one of the Verblunsky coefficients
results in a perturbation of twice the minimal rank. The explicit
form for the minimal rank perturbation and the resulting two half-lattice CMV
matrices are obtained. In addition, formulas relating the Weyl--Titchmarsh
$m$-functions (resp., matrices) associated with the involved CMV operators and their Green's functions (resp., matrices) are derived.
%
%
%
\end{abstract}

\section{Introduction}\lb{s1}

CMV operators are a special class of unitary semi-infinite or doubly-infinite five-diagonal matrices which received enormous attention in recent years. We refer to \eqref{2.8} and \eqref{3.18} for the explicit form of doubly infinite CMV operators on $\bbZ$ in the case of scalar, respectively, matrix-valued Verblunsky coefficients. For the corresponding half-lattice CMV operators we refer to \eqref{2.14} and \eqref{3.26}.

The actual history of CMV operators (with scalar Verblunsky coefficients) is somewhat  intriguing: The corresponding unitary semi-infinite five-diagonal matrices were first introduced in 1991 by Bunse--Gerstner and Elsner \cite{BGE91}, and subsequently discussed in detail by Watkins \cite{Wa93} in 1993 (cf.\ the discussion in Simon
\cite{Si06}). They were subsequently rediscovered by Cantero, Moral, and Vel\'azquez (CMV) in \cite{CMV03}. In \cite[Sects.\ 4.5, 10.5]{Si04}, Simon introduced the corresponding notion of unitary doubly infinite five-diagonal matrices and coined the term ``extended'' CMV matrices. For simplicity, we will just speak of CMV operators, irrespective of whether or not they are half-lattice or full-lattice operators. We also note that in a context different from orthogonal polynomials on the unit circle, Bourget, Howland, and Joye \cite{BHJ03} introduced a family of doubly infinite matrices with three sets of parameters which, for special choices of the parameters, reduces to two-sided CMV matrices on $\bbZ$. Moreover, it is possible to connect unitary block Jacobi matrices to the trigonometric moment problem (and hence to CMV matrices) as discussed by Berezansky and Dudkin \cite{BD05}, \cite{BD06}.

The relevance of this unitary operator $\U$ on $\ell^2(\bbZ)^m$, more precisely, the relevance of the corresponding half-lattice CMV operator $\U_{+,0}$ in
$\ell^2(\bbN_0)^m$ is derived from its intimate relationship with the trigonometric moment problem and hence with finite measures on the unit circle $\dD$. (Here
$\bbN_0=\bbN\cup\{0\}$.) Following \cite{CGZ07}, \cite{CGZ08}, \cite{GZ06},
\cite{GZ06a}, and \cite{Zi08}, this will be reviewed in some detail, and also extended in certain respects, in Sections \ref{s2} and \ref{s3}, as this material is of fundamental importance to the principal topics (such as decoupling of full-lattice CMV operators into direct sums of left and right half-lattice CMV operators and a similar result for associated Green's functions) discussed in this paper, but we also refer to the monumental two-volume treatise by Simon \cite{Si04} (see also \cite{Si04b} and \cite{Si05}) and the exhaustive bibliography therein. For classical results on orthogonal polynomials on the unit circle we refer, for instance, to \cite{Ak65}, \cite{Ge46}--\cite{Ge61}, \cite{Kr45},
\cite{Sz20}--\cite{Sz78}, \cite{Ve35}, \cite{Ve36}. More recent references relevant to the spectral theoretic content of this paper are \cite{De07}, \cite{GJ96}--\cite{GT94},
\cite{GZ06}, \cite{GZ06a}, \cite{GN01}, \cite{PY04}, \cite{Si04a}, and \cite{Zi08}. The
full-lattice CMV operators $\U$ on $\bbZ$ are closely related to an important, and only recently intensively studied, completely integrable nonabelian version of the defocusing nonlinear Schr\"odinger equation (continuous in time but discrete in space), a special case of the Ablowitz--Ladik system. Relevant references in this context are, for instance,
\cite{AL75}--\cite{APT04}, \cite{GGH05}, \cite{GH05}--\cite{GHMT07b}, \cite{Li05},
\cite{MEKL95}--\cite{Ne06}, \cite{Sc89}, \cite{Ve99}, and the literature cited therein. We emphasize that the case of matrix-valued coefficients $\alpha_k$ is considerably less studied than the case of scalar coefficients.

We should also emphasize that while there is an extensive literature on orthogonal matrix-valued polynomials on the real line and on the unit circle, we refer, for instance, to
\cite{AN84}, \cite{BC92}, \cite[Ch.\ VII]{Be68}, \cite{BG90}, \cite{CFMV03}, \cite{CG06},
\cite{DG92}--\cite{DV95}, \cite{Ge81}, \cite{Ge82}, \cite{Kr49}, \cite{Kr71}, \cite{Le47},
\cite{Lo99}, \cite{Os97}--\cite{Os02}, \cite{Ro90}, \cite{Si06a}, \cite{YM01}--\cite{YK78}, and the large body of literature therein, the case of CMV operators with matrix-valued Verblunsky coefficients appears to be a much less explored frontier. The only references we are aware of in this context are Simon's treatise \cite[Part 1, Sect.\ 2.13]{Si04} and the recent papers \cite{Ar08}, \cite{CGZ07}, \cite{DPS08}, and \cite{Si06}.

Finally, a brief description of the content of each section in this paper: In Section \ref{s2} we review, and in part, extend the basic Weyl--Titchmarsh theory for half-lattice CMV operators with scalar Verblunsky coefficients originally derived in \cite{GZ06}, and recall its intimate connections with transfer matrices and orthogonal Laurent polynomials. The principal result of this section, Theorem \ref{t2.3}, then provides a necessary and sufficient condition for the difference between the full-lattice CMV operator $U$ and its ``decoupling'' into a direct sum of appropriate left and right half-lattice CMV operators to be of rank one. The same result is also derived for the resolvent differences of $U$ and the resolvent of its decoupling into a direct sum of left and right half-lattice CMV operators. Theorem \ref{t2.3} is in sharp contrast to the familiar Jacobi case, since decoupling a full-lattice CMV matrix by changing one of the Verblunsky coefficients
results in a perturbation of rank two. While this difference compared to Jacobi operators was noticed first by Simon \cite[Sect.\ 4.5]{Si04}, we explore it further here and provide a complete discussion of this decoupling phenomenon, including its extension to the matrix-valued case, which represents a new result. We conclude this section with a discussion of half-lattice Green's functions in Lemma \ref{l2.6}, extending a result in
\cite{GZ06}.

In Section \ref{s3} we develop all these results for CMV operators with
$m\times m$, $m\in\bbN$, matrix-valued Verblunsky coefficients. In particular, in
Theorem \ref{t3.6}, the principal result of this section, we provide
a necessary and sufficient condition for the difference between the
full-lattice CMV operator $U$ and its decoupling into a direct sum of
appropriate left and right half-lattice CMV operators to be of minimal
rank $m$.

Finally, Appendix \ref{sA} summarizes basic facts on matrix-valued
Caratheodory and Schur functions relevant to this paper.

\section{CMV operators with scalar coefficients} \lb{s2}

This section is devoted to a study of CMV operators associated with
scalar Verblunsky coefficients. We derive a criterion under which
a difference of a full-lattice CMV operator and a direct sum of two
half-lattice CMV operators is of rank one. The same condition will
also imply a similar result for the resolvents of these operators. At
the end of the section we establish relations that hold between
Weyl--Titchmarsh $m$-functions associated with the above operators
and derive explicit expressions for half-lattice Green's matrices.

We start by introducing basic notations used throughout this paper.
Let $\s{\Z}$ be the space of complex-valued sequences and
$\lt{\Z}\subset \s{\Z}$ be the usual Hilbert space of all square
summable complex-valued sequences with scalar product
$(\cdot,\cdot)_{\lt{\Z}}$ linear in the second argument. The {\it
standard basis} in $\lt{\Z}$ is denoted by
\begin{equation}
\{\delta_k\}_{k\in\Z}, \quad
\delta_k=(\dots,0,\dots,0,\underbrace{1}_{k},0,\dots,0,\dots)^\top,
\; k\in\Z.
\end{equation}

For $m\in\N$ and $J\subseteq\bbR$ an interval, we will identify
$\oplus_{j=1}^m\lt{J\cap\Z}$ and $\lt{J\cap\Z}\otimes\C^m$ and then
use the simplified notation $\ltm{J\cap\Z}$. For simplicity, the
identity operators on $\lt{J\cap\Z}$ and $\ltm{J\cap\Z}$ are
abbreviated by $I$ without separately indicating its dependence on
$m$ and $J$. The identity $m\times m$ matrix is denoted by $I_m$.

Throughout this section we make the following basic assumption:

\begin{hypothesis} \lb{h2.1}
Let $\alpha=\{\al_k\}_{k\in\Z}\in \s{\Z}$ be a sequence of complex
numbers such that
\begin{equation} \lb{2.2}
\al_k\in\D, \quad k \in \Z.
\end{equation}
\end{hypothesis}

Given a sequence $\alpha$ satisfying \eqref{2.2}, we define the
following sequence of positive real numbers $\{\rho_k\}_{k\in\Z}$ by
\begin{equation}
\rho_k = \big[1-\abs{\al_k}^2\big]^{1/2}, \quad k \in \Z. \lb{2.3}
\end{equation}

Following Simon \cite{Si04}, we call $\{\al_k\}_{k\in\Z}$ the
Verblunsky coefficients in honor of Verblunsky's pioneering work in
the theory of orthogonal polynomials on the unit circle \cite{Ve35},
\cite{Ve36}.

Next, we also introduce a sequence of $2\times 2$ unitary matrices
$\Te_k$ by
\begin{equation} \lb{2.4}
\Te_k = \begin{pmatrix} -\al_k & \rho_k \\ \rho_k & \ol{\al_k}
\end{pmatrix},
\quad k \in \Z,
\end{equation}
and two unitary operators $V$ and $W$ on $\lt{\Z}$ by their matrix
representations in the standard basis of $\lt{\Z}$ as follows,
\begin{align} \lb{2.5}
V &= \begin{pmatrix} \ddots & & &
\raisebox{-3mm}[0mm][0mm]{\hspace*{-5mm}\Huge $0$}  \\ & \Te_{2k-2} &
& \\ & & \Te_{2k} & & \\ &
\raisebox{0mm}[0mm][0mm]{\hspace*{-10mm}\Huge $0$} & & \ddots
\end{pmatrix}, \quad
W = \begin{pmatrix} \ddots & & &
\raisebox{-3mm}[0mm][0mm]{\hspace*{-5mm}\Huge $0$}
\\ & \Te_{2k-1} &  &  \\ &  & \Te_{2k+1} &  & \\ &
\raisebox{0mm}[0mm][0mm]{\hspace*{-10mm}\Huge $0$} & & \ddots
\end{pmatrix},
\end{align}
where
\begin{align}
\begin{pmatrix}
V_{2k-1,2k-1} & V_{2k-1,2k} \\
V_{2k,2k-1}   & V_{2k,2k}
\end{pmatrix} =  \Te_{2k},
\quad
\begin{pmatrix}
W_{2k,2k} & W_{2k,2k+1} \\ W_{2k+1,2k}  & W_{2k+1,2k+1}
\end{pmatrix} =  \Te_{2k+1},
\quad k\in\Z.
\end{align}
Moreover, we introduce the unitary operator $U$ on $\lt{\Z}$ by
\begin{equation} \lb{2.7}
U = VW,
\end{equation}
or in matrix form, in the standard basis of $\lt{\Z}$, by
\begin{align}
U &= \begin{pmatrix} \ddots &&\hspace*{-8mm}\ddots
&\hspace*{-10mm}\ddots &\hspace*{-12mm}\ddots &\hspace*{-14mm}\ddots
&&& \raisebox{-3mm}[0mm][0mm]{\hspace*{-6mm}{\Huge $0$}}
\\
&0& -\al_{0}\rho_{-1} & -\ol{\al_{-1}}\al_{0} & -\al_{1}\rho_{0} &
\rho_{0}\rho_{1}
\\
&& \rho_{-1}\rho_{0} &\ol{\al_{-1}}\rho_{0} & -\ol{\al_{0}}\al_{1} &
\ol{\al_{0}}\rho_{1} & 0
\\
&&&0& -\al_{2}\rho_{1} & -\ol{\al_{1}}\al_{2} & -\al_{3}\rho_{2} &
\rho_{2}\rho_{3}
\\
&&\raisebox{-4mm}[0mm][0mm]{\hspace*{-6mm}{\Huge $0$}} &&
\rho_{1}\rho_{2} & \ol{\al_{1}}\rho_{2} & -\ol{\al_{2}}\al_{3} &
\ol{\al_{2}}\rho_{3}&0
\\
&&&&&\hspace*{-14mm}\ddots &\hspace*{-14mm}\ddots
&\hspace*{-14mm}\ddots &\hspace*{-8mm}\ddots &\ddots
\end{pmatrix}     \lb{2.8}
\\
&= \rho^- \rho \, \deven \, S^{--} + ({\ol \alpha}^-\rho \, \deven -
\alpha^+\rho \, \dodd) S^-  - {\ol \alpha}\alpha^+ \no
\\
& \quad + ({\ol \alpha} \rho^+ \, \deven - \alpha^{++} \rho^+ \,
\dodd) S^+ + \rho^+ \rho^{++} \, \dodd \, S^{++},\lb{2.9}
\end{align}
where $\deven$ and $\dodd$ denote the characteristic functions of the
even and odd integers,
\begin{equation}\lb{2.10}
\deven = \chi_{_{2\Z}}, \quad \dodd = 1 - \deven = \chi_{_{2\Z +1}}
\end{equation}
and $S^\pm$, $S^{++}$, $S^{--}$ denote the shift operators acting
upon $\s{\Z}$, that is, $S^{\pm}f(\cdot)=f^{\pm}(\cdot)=f
(\cdot\pm1)$ for $f\in \s{\Z}$, $S^{++}=S^+S^+$, and $S^{--}=S^-S^-$.
Here the diagonal entries in the infinite matrix \eqref{2.8} are
given by $U_{k,k}=-\ol{\alpha_k}\alpha_{k+1}$, $k\in\Z$.

As explained in the introduction, in the recent literature on
orthogonal polynomials on the unit circle, such operators $U$ are
frequently called CMV operators.

Next we recall some of the principal results of \cite{GZ06} needed in
this paper.

\begin{lemma} \lb{l2.2}
Let $z\in\C\backslash\{0\}$ and suppose $\{u(z,k)\}_{k\in\Z}$,
$\{v(z,k)\}_{k\in\Z}\in\s{\Z}$. Then the following items
$(i)$--$(iii)$ are equivalent:
\begin{align}
(i) &\quad (U u(z,\cdot))(k) = z u(z,k), \quad (W u(z,\cdot))(k)=z
v(z,k), \quad k\in\Z.
\\
(ii) &\quad (W u(z,\cdot))(k)=z v(z,k), \quad (V v(z,\cdot))(k) =
u(z,k), \quad k\in\Z.
\\
(iii) &\quad \binom{u(z,k)}{v(z,k)} = T(z,k)
\binom{u(z,k-1)}{v(z,k-1)}, \quad k\in\Z,  \lb{2.11}
\end{align}
where the transfer matrices $T(z,k)$, $z\in\Cz$, $k\in\Z$, are given
by
\begin{equation}
T(z,k) = \begin{cases} \frac{1}{\rho_{k}} \begin{pmatrix} \al_{k} & z
\\ 1/z & \ol{\al_{k}} \end{pmatrix},  & \text{$k$
odd,}  \\[20pt] \frac{1}{\rho_{k}} \begin{pmatrix} \ol{\al_{k}} & 1 \\
1 & \al_{k} \end{pmatrix}, & \text{$k$ even.} \end{cases} \lb{2.12}
\end{equation}
Here $U$, $V$, and $W$ are understood in the sense of difference
expressions on $\s{\Z}$ rather than difference operators on
$\lt{\Z}$.
\end{lemma}

If one sets $\al_{k_0} = e^{it}$, $t\in [0,2\pi)$, for some reference
point $k_0\in\Z$, then the CMV operator (denoted in this case by
$U^{(t)}_{k_0}$) splits into a direct sum of two half-lattice
operators $U_{-,k_0-1}^{(t)}$ and $U_{+,k_0}^{(t)}$ acting on
$\lt{(-\infty,k_0-1]\cap\Z}$ and on $\lt{[k_0,\infty)\cap\Z}$,
respectively. Explicitly, one obtains
\begin{align}
\begin{split}
& U^{(t)}_{k_0}=U^{(t)}_{-,k_0-1} \oplus U^{(t)}_{+,k_0} \, \text{ on
} \, \lt{(-\infty,k_0-1]\cap\Z} \oplus \lt{[k_0,\infty)\cap\Z}
\\ & \text{if } \, \al_{k_0} = e^{it}, \; t\in [0,2\pi).
\lb{2.13}
\end{split}
\end{align}
(Strictly, speaking, setting $\al_{k_0} = e^{it}$, $t\in[0,2\pi)$,
for some reference point $k_0\in\Z$ contradicts our basic Hypothesis
\ref{h2.1}. However, as long as the exception to Hypothesis
\ref{h2.1} refers to only one site, we will safely ignore this
inconsistency in favor of the notational simplicity it provides by
avoiding the introduction of a properly modified hypothesis on
$\{\alpha_k\}_{k\in\Z}$.) Similarly, one obtains $V^{(t)}_{k_0}$,
$W^{(t)}_{k_0}$, $V^{(t)}_{\pm,k_0}$, and $W^{(t)}_{\pm,k_0}$, so
that
\begin{align}
& V^{(t)}_{k_0}=V^{(t)}_{-,k_0-1} \oplus V^{(t)}_{+,k_0}, &&
W^{(t)}_{k_0}=W^{(t)}_{-,k_0-1} \oplus W^{(t)}_{+,k_0}, \no
\\
& U^{(t)}_{k_0}=V^{(t)}_{k_0} W^{(t)}_{k_0}, && U^{(t)}_{\pm,k_0} =
V^{(t)}_{\pm,k_0} W^{(t)}_{\pm,k_0}. \lb{2.14}
\end{align}
For simplicity we will abbreviate
\begin{equation}
U_{\pm,k_0} =
U_{\pm,k_0}^{(t=0)}=V_{\pm,k_0}^{(t=0)}W_{\pm,k_0}^{(t=0)}
=V_{\pm,k_0} W_{\pm,k_0}.  \lb{2.15}
\end{equation}

It is instructive to introduce one more sequence of Verblunsky
coefficients $\be=\{\be_k\}_{k\in\Z}\in\s{\Z}$ by
\begin{align}
\be_k=e^{-it}\al_k, \quad k\in\Z, \lb{2.17}
\end{align}
so that the sequence $\{\rho_k\}_{k\in\Z}$ is unchanged and
$\al_{k_0}=e^{it}$ corresponds to $\be_{k_0}=1$. Then the CMV
operators $U_\be$ and $U_{\pm,k_0;\be}$ associated with $\be$ are
unitarily equivalent to the corresponding CMV operators $U_\al$ and
$U^{(t)}_{\pm,k_0;\al}$ associated with $\al$. Indeed, one verifies
that
\begin{align}
\begin{pmatrix}e^{-it/2}&0\\0&e^{it/2}\end{pmatrix}
\begin{pmatrix} -\al_k & \rho_k \\ \rho_k & \ol{\al_k}\end{pmatrix}
\begin{pmatrix}e^{-it/2}&0\\0&e^{it/2}\end{pmatrix} =
\begin{pmatrix} -\be_k & \rho_k \\ \rho_k & \ol{\be_k}\end{pmatrix},
\quad k\in\Z, \lb{2.17a}
\end{align}
and hence, setting $A$ to be the following diagonal unitary operator
on $\lt{\Z}$,
\begin{align}
A = e^{-it/2}\dodd + e^{it/2}\deven,
\end{align}
one obtains for the full-lattice and the direct sum of half-lattice
CMV operators,
\begin{align}
& AU_{\al}A^* = [AV_{\al}A] [A^*W_{\al}A^*] = V_{\be}W_{\be} =
U_{\be}, \lb{2.17b}
\\
& AU^{(t)}_{k_0;\al}A^* = \big[AV^{(t)}_{k_0;\al}A\big]
\big[A^*W^{(t)}_{k_0;\al}A^*\big] = V_{k_0;\be}W_{k_0;\be} = U_{k_0;\be}.
\lb{2.17c}
\end{align}
We refer to \cite[Sect. 3]{GZ06a} for additional results on CMV
operators with Verblunsky coefficients related via \eqref{2.17}.

Now we turn to our principal result of this section.

\begin{theorem} \lb{t2.3}
Fix $t_1, t_2\in[0,2\pi)$, $k_0\in\Z$, $z\in\CdD$, and let
$U^{(t_1,t_2)}_{k_0}$ denote the following unitary operator on
$\lt{\Z}$,
\begin{align}
U^{(t_1,t_2)}_{k_0} = U^{(t_1)}_{-,k_0-1} \oplus U^{(t_2)}_{+,k_0}.
\lb{2.18}
\end{align}
Then $U-U^{(t_1,t_2)}_{k_0}$ and
$(U-zI)^{-1}-\big(U^{(t_1,t_2)}_{k_0}-zI\big)^{-1}$ are of rank one if and
only if the relation $t_1=2\arg\big[i(\al_{k_0}e^{-it_2/2}-e^{it_2/2})\big]$ 
holds. Otherwise, these differences are of rank two. In particular,
$U-U^{(t)}_{k_0}$ and $(U-zI)^{-1}-\big(U^{(t)}_{k_0}-zI\big)^{-1}$ are of
rank two for any $t\in[0,2\pi)$.
\end{theorem}
\begin{proof}
Similar to \eqref{2.18} we introduce unitary operators
$V^{(t_1,t_2)}_{k_0}$ and $W^{(t_1,t_2)}_{k_0}$ by
\begin{align}
V^{(t_1,t_2)}_{k_0} = V^{(t_1)}_{-,k_0-1} \oplus V^{(t_2)}_{+,k_0}
\,\text{ and }\, W^{(t_1,t_2)}_{k_0} = W^{(t_1)}_{-,k_0-1} \oplus
W^{(t_2)}_{+,k_0} \,\text{ on }\, \lt{\Z}. \lb{2.19}
\end{align}
Then $U^{(t_1,t_2)}_{k_0}=V^{(t_1,t_2)}_{k_0}W^{(t_1,t_2)}_{k_0}$ by
\eqref{2.14}, and hence, it follows from \eqref{2.5} that
\begin{align}
U-U^{(t_1,t_2)}_{k_0} =
\begin{cases}
V\big(W-W^{(t_1,t_2)}_{k_0}\big), & \text{$k_0$ odd,} \\
\big(V-V^{(t_1,t_2)}_{k_0}\big)W, & \text{$k_0$ even.}
\end{cases} \lb{2.20}
\end{align}
For $k_0$ odd, $D=W-W^{(t_1,t_2)}_{k_0}$ is block-diagonal with all
its $2\times 2$ blocks on the diagonal being zero except for one
which has the following form
\begin{align}
\begin{pmatrix}
D_{k_0-1,k_0-1} & D_{k_0-1,k_0} \\ D_{k_0,k_0-1}  & D_{k_0,k_0}
\end{pmatrix}
=
\begin{pmatrix}
-\al_{k_0} & \rho_0 \\ \rho_{k_0} & \ol{\al_{k_0}}
\end{pmatrix}
-
\begin{pmatrix}
-e^{it_1} & 0 \\ 0 & e^{-it_2}
\end{pmatrix}. \lb{2.21}
\end{align}
Thus, the difference $U-U^{(t_1,t_2)}_{k_0}$ in \eqref{2.20} is
always of rank one or two and it is precisely of rank one if and only
if the $2\times 2$ matrix in \eqref{2.21} is of rank 1. The latter
case is equivalent to
\begin{align}
0&=\det\begin{pmatrix} D_{k_0-1,k_0-1} & D_{k_0-1,k_0} \\
D_{k_0,k_0-1} & D_{k_0,k_0}
\end{pmatrix}
= e^{it_1}\ol{\al_{k_0}} + e^{-it_2}\al_{k_0} - e^{i(t_1-t_2)} - 1
\no
\\
&= (\ol{\al_{k_0}}-e^{-it_2})\left(e^{it_1} + e^{-it_2}
\f{\al_{k_0}-e^{it_2}}{\ol{\al_{k_0}}-e^{-it_2}}\right) \no
\\
&= (\ol{\al_{k_0}}-e^{-it_2})\left(e^{it_1} -
\f{i\al_{k_0}e^{-it_2/2}-e^{it_2/2}}
{\ol{i\al_{k_0}e^{-it_2/2}-e^{it_2/2}}}\right),\lb{2.22a}
\end{align}
which holds if and only if
$t_1=2\arg\big[i(\al_{k_0}e^{-it_2/2}-e^{it_2/2})\big]$. The case of
even $k_0$ follows similarly.

Finally, the statement for the resolvents follows from the result
for $U-U^{(t_1,t_2)}_{k_0}$ and the following identity,
\begin{align}
(U-zI)^{-1}-(U^{(t_1,t_2)}_{k_0}-zI)^{-1} = -(U-zI)^{-1}
\Big[U-U^{(t_1,t_2)}_{k_0}\Big] (U^{(t_1,t_2)}_{k_0}-zI)^{-1}.
\end{align}
\end{proof}

Next we present formulas that link various spectral theoretic objects
associated with half-lattice CMV operators $U^{(t)}_{\pm,k_0}$ for
different values of $t\in[0,2\pi)$. We start with an analog of Lemma
\ref{l2.2} for difference expressions $U^{(t)}_{\pm,k_0}$,
$V^{(t)}_{\pm,k_0}$, and $W^{(t)}_{\pm,k_0}$. In the special case
$t=0$ it is proven in \cite[Lem. 2.3]{GZ06} and the general case
below follows immediately from the special case and the observation
of unitary equivalence in \eqref{2.17c}.

\begin{lemma} \lb{l2.3}
Fix $t\in[0,2\pi)$, $k_0\in\Z$, $z\in\C\backslash\{0\}$, and let
$\big\{\hat p^{(t)}_+(z,k,k_0)\big\}_{k\geq k_0}$,
$\big\{\hat r^{(t)}_+(z,k,k_0)\big\}_{k\geq k_0}\in\s{[k_0,\infty)\cap\Z}$. Then the
following items $(i)$--$(iii)$ are equivalent:
\begin{align}
(i) &\quad \big(U^{(t)}_{+,k_0} \hat p^{(t)}_+(z,\cdot,k_0)\big)(k) = z \hat
p^{(t)}_+(z,k,k_0), \no \\
&\quad \big(W^{(t)}_{+,k_0} \hat p^{(t)}_+(z,\cdot,k_0)\big)(k) = z \hat
r^{(t)}_+(z,k,k_0), \quad k\geq k_0.
\\
(ii) &\quad \big(W^{(t)}_{+,k_0} \hat p^{(t)}_+(z,\cdot,k_0)\big)(k) = z
\hat r^{(t)}_+(z,k,k_0), \no \\
&\quad \big(V^{(t)}_{+,k_0} \hat r^{(t)}_+(z,\cdot,k_0)\big)(k) = \hat
p^{(t)}_+(z,k,k_0), \quad k\geq k_0.
\\
(iii) &\quad \binom{\hat p^{(t)}_+(z,k,k_0)}{\hat r^{(t)}_+(z,k,k_0)}
= T(z,k) \binom{\hat p^{(t)}_+(z,k-1,k_0)}{\hat
r^{(t)}_+(z,k-1,k_0)}, \quad k > k_0, \no \\
&\quad \hat p^{(t)}_+(z,k_0,k_0) =
\begin{cases}
ze^{it} \hat r^{(t)}_+(z,k_0,k_0), & \text{$k_0$ odd}, \\
e^{-it} \hat r^{(t)}_+(z,k_0,k_0), & \text{$k_0$ even}.
\end{cases}
\end{align}
Similarly, let $\big\{\hat p^{(t)}_-(z,k,k_0)\big\}_{k\leq k_0}$,
$\big\{\hat r^{(t)}_-(z,k,k_0)\big\}_{k\leq k_0}\in\s{(-\infty,k_0]\cap\Z}$. Then the
following items $(iv)$--$(vi)$are equivalent:
\begin{align}
(iv) &\quad \big(U^{(t)}_{-,k_0} \hat p^{(t)}_-(z,\cdot,k_0)\big)(k) = z
\hat p^{(t)}_-(z,k,k_0), \no \\
&\quad \big(W^{(t)}_{-,k_0} \hat p^{(t)}_-(z,\cdot,k_0)\big)(k) = z \hat
r^{(t)}_-(z,k,k_0), \quad k\leq k_0.
\\
(v) &\quad \big(W^{(t)}_{-,k_0} \hat p^{(t)}_-(z,\cdot,k_0)\big)(k) = z
\hat r^{(t)}_-(z,k,k_0), \no \\
&\quad \big(V^{(t)}_{-,k_0} \hat r^{(t)}_-(z,\cdot,k_0)\big)(k) = \hat
p^{(t)}_-(z,k,k_0), \quad k\leq k_0.
\\
(vi) &\quad \binom{\hat p^{(t)}_-(z,k-1,k_0)}{\hat
r^{(t)}_-(z,k-1,k_0)}=T(z,k)^{-1} \binom{\hat
p^{(t)}_-(z,k,k_0)}{\hat r^{(t)}_-(z,k,k_0)}, \quad k \leq k_0,\no
\\
&\quad \hat p^{(t)}_-(z,k_0,k_0) =
\begin{cases}
-e^{it}\hat r^{(t)}_-(z,k_0,k_0), & \text{$k_0$ odd,}
\\
-ze^{-it} \hat r^{(t)}_-(z,k_0,k_0), & \text{$k_0$ even.}
\end{cases}
\end{align}
\end{lemma}

In the following, we denote by
$\Big(\begin{smallmatrix}p^{(t)}_\pm(z,k,k_0)\\
r^{(t)}_\pm(z,k,k_0)\end{smallmatrix}\Big)_{k\in\Z}$ and
$\Big(\begin{smallmatrix}q^{(t)}_\pm(z,k,k_0)\\
s^{(t)}_\pm(z,k,k_0)\end{smallmatrix}\Big)_{k\in\Z}$,
$z\in\C\backslash\{0\}$, four linearly independent solutions of
\eqref{2.11} with the initial conditions:
\begin{align}
\binom{p^{(t)}_+(z,k_0,k_0)}{r^{(t)}_+(z,k_0,k_0)} = \begin{cases}
\binom{z\ol{\ga}}{\ga}, &
\text{$k_0$ odd,} \\[1mm]
\binom{\ga}{\ol{\ga}}, & \text{$k_0$ even,} \end{cases} \quad
\binom{q^{(t)}_+(z,k_0,k_0)}{s^{(t)}_+(z,k_0,k_0)} = \begin{cases}
\binom{z\ol{\ga}}{-\ga}, &
\text{$k_0$ odd,} \\[1mm]
\binom{-\ga}{\ol{\ga}}, & \text{$k_0$ even.} \end{cases} \lb{2.22}
\\
\binom{p^{(t)}_-(z,k_0,k_0)}{r^{(t)}_-(z,k_0,k_0)} = \begin{cases}
\binom{\ol{\ga}}{-\ga}, &
\text{$k_0$ odd,} \\[1mm]
\binom{-z\ga}{\ol{\ga}}, & \text{$k_0$ even,} \end{cases} \quad
\binom{q^{(t)}_-(z,k_0,k_0)}{s^{(t)}_-(z,k_0,k_0)} =
\begin{cases} \binom{\ol{\ga}}{\ga}, &
\text{$k_0$ odd,} \\[1mm]
\binom{z\ga}{\ol{\ga}}, & \text{$k_0$ even,} \end{cases} \lb{2.23}
\end{align}
where $\ga=e^{-it/2}$. Then it follows that $p^{(t)}_\pm(z,k,k_0)$,
$q^{(t)}_\pm(z,k,k_0)$, $r^{(t)}_\pm(z,k,k_0)$, and
$s^{(t)}_\pm(z,k,k_0)$, $k,k_0\in\Z$, are Laurent polynomials in $z$,
that is, finite linear combinations of terms $z^k$, k$\in\Z$, with
complex-valued coefficients.

Since all of the above sequences satisfy the same recursion relation
\eqref{2.11} which can have at most two linearly independent
solutions, these sequences satisfy various identities. Some of them
we state in the following lemma.

\begin{lemma} \lb{l2.4}
Let $t_1,t_2\in[0,2\pi)$ and $\ga_j=e^{-it_j/2}$, $j=1,2$. Then
\begin{align}
&\binom{q^{(t_2)}_\pm(z,\cdot,k_0)}{s^{(t_2)}_\pm(z,\cdot,k_0)} =
\Re(\ga_1\ol{\ga_2})
\binom{q^{(t_1)}_\pm(z,\cdot,k_0)}{s^{(t_1)}_\pm(z,\cdot,k_0)} +
i\Im(\ga_1\ol{\ga_2})
\binom{p^{(t_1)}_\pm(z,\cdot,k_0)}{r^{(t_1)}_\pm(z,\cdot,k_0)},
\lb{2.24}
\\
&\binom{p^{(t_2)}_\pm(z,\cdot,k_0)}{r^{(t_2)}_\pm(z,\cdot,k_0)} =
i\Im(\ga_1\ol{\ga_2})
\binom{q^{(t_1)}_\pm(z,\cdot,k_0)}{s^{(t_1)}_\pm(z,\cdot,k_0)} +
\Re(\ga_1\ol{\ga_2})
\binom{p^{(t_1)}_\pm(z,\cdot,k_0)}{r^{(t_1)}_\pm(z,\cdot,k_0)},
\lb{2.25}
\\
&\binom{q^{(t_2)}_-(z,\cdot,k_0)}{s^{(t_2)}_-(z,\cdot,k_0)} =
\frac{\ga_1\ol{\ga_2}-\ol{\ga_1}\ga_2z}{2z^{k_0 \, ({\rm mod}\,2)}}
\binom{q^{(t_1)}_+(z,\cdot,k_0)}{s^{(t_1)}_+(z,\cdot,k_0)} +
\frac{\ga_1\ol{\ga_2}+\ol{\ga_1}\ga_2z}{2z^{k_0 \, ({\rm mod}\,2)}}
\binom{p^{(t_1)}_+(z,\cdot,k_0)}{r^{(t_1)}_+(z,\cdot,k_0)},
\label{2.26}
\\
&\binom{p^{(t_2)}_-(z,\cdot,k_0)}{r^{(t_2)}_-(z,\cdot,k_0)} =
\frac{\ga_1\ol{\ga_2}+\ol{\ga_1}\ga_2z}{2z^{k_0 \, ({\rm mod}\,2)}}
\binom{q^{(t_1)}_+(z,\cdot,k_0)}{s^{(t_1)}_+(z,\cdot,k_0)} +
\frac{\ga_1\ol{\ga_2}-\ol{\ga_1}\ga_2z}{2z^{k_0 \, ({\rm mod}\,2)}}
\binom{p^{(t_1)}_+(z,\cdot,k_0)}{r^{(t_1)}_+(z,\cdot,k_0)},
\lb{2.26a}
\\
&\binom{q^{(t_2)}_-(z,\cdot,k_0-1)}{s^{(t_2)}_-(z,\cdot,k_0-1)} =
\frac{i\Im(\ga_1\ol{\ga_2}+\al_{k_0}\ga_1\ga_2)}{\rho_{k_0}}
\binom{q^{(t_1)}_+(z,\cdot,k_0)}{s^{(t_1)}_+(z,\cdot,k_0)} \no
\\ &\hspace{35mm}
+ \frac{\Re(\ga_1\ol{\ga_2}+\al_{k_0}\ga_1\ga_2)}{\rho_{k_0}}
\binom{p^{(t_1)}_+(z,\cdot,k_0)}{r^{(t_1)}_+(z,\cdot,k_0)},
\label{2.27}
\\
&\binom{p^{(t_2)}_-(z,\cdot,k_0-1)}{r^{(t_2)}_-(z,\cdot,k_0-1)} =
\frac{\Re(\ga_1\ol{\ga_2}-\al_{k_0}\ga_1\ga_2)}{\rho_{k_0}}
\binom{q^{(t_1)}_+(z,\cdot,k_0)}{s^{(t_1)}_+(z,\cdot,k_0)} \no
\\ &\hspace{35mm}
+ \frac{i\Im(\ga_1\ol{\ga_2}-\al_{k_0}\ga_1\ga_2)}{\rho_{k_0}}
\binom{p^{(t_1)}_+(z,\cdot,k_0)}{r^{(t_1)}_+(z,\cdot,k_0)}.
\lb{2.27a}
\end{align}
In particular, whenever
$t_1=2\arg\big[i(\al_{k_0}e^{-it_2/2}-e^{it_2/2})\big]$ identity
\eqref{2.27a} simplifies to
\begin{align}
\binom{p^{(t_2)}_-(z,\cdot,k_0-1)}{r^{(t_2)}_-(z,\cdot,k_0-1)} &=
\frac{i\abs{e^{it_2}-\al_{k_0}}}{\rho_{k_0}}
\binom{p^{(t_1)}_+(z,\cdot,k_0)}{r^{(t_1)}_+(z,\cdot,k_0)}. \lb{2.28}
\end{align}
\end{lemma}
\begin{proof}
Since both sides of \eqref{2.24}--\eqref{2.27a} satisfy the same
recursion relation \eqref{2.11} it suffices to check these equalities
only at one point, say at point $k=k_0$. Substituting \eqref{2.22}
and \eqref{2.23} into \eqref{2.24}--\eqref{2.26a} one verifies the
first four identities. Using \eqref{2.22} and \eqref{2.23} once again
and applying transfer matrix $T(z,k_0)$ to the left hand-sides of
\eqref{2.27}--\eqref{2.28} one verifies the last three identities.
\end{proof}

Next, following \cite{GZ06} we introduce half-lattice
Weyl--Titchmarsh $m$-functions associated with the CMV operators
$U^{(t)}_{\pm,k_0}$ by
\begin{align}
\begin{split}
m^{(t)}_\pm(z,k_0) &= \pm
(\delta_{k_0},(U^{(t)}_{\pm,k_0}+zI)(U^{(t)}_{\pm,k_0}-zI)^{-1}
\delta_{k_0})_{\lt{\Z\cap[k_0,\pm\infty)}} \\
& =\pm \oint_\dD d\mu^{(t)}_{\pm}(\zeta,k_0)\,
\frac{\zeta+z}{\zeta-z}, \quad z\in\C\backslash\dD,
\end{split} \lb{2.30}
\end{align}
where
\begin{equation}
d\mu^{(t)}_\pm(\zeta,k_0) = d(\de_{k_0},E_{U^{(t)}_{\pm,k_0}}(\zeta)
\de_{k_0})_{\lt{\Z\cap[k_0,\pm\infty)}}, \quad \zeta\in\dD, \lb{2.31}
\end{equation}
and $dE_{U^{(t)}_{\pm,k_0}}(\cdot)$ denote the operator-valued
spectral measures of the operators $U^{(t)}_{\pm,k_0}$,
\begin{equation}
U^{(t)}_{\pm,k_0}=\oint_{\dD} dE_{U^{(t)}_{\pm,k_0}}(\zeta)\,\zeta.
\end{equation}
Then following the steps of \cite[Cor. 2.14]{GZ06} one verifies that
\begin{equation}
\binom{q^{(t)}_\pm(z,\cdot,k_0)}{s^{(t)}_\pm(z,\cdot,k_0)} +
m^{(t)}_\pm(z,k_0)
\binom{p^{(t)}_\pm(z,\cdot,k_0)}{r^{(t)}_\pm(z,\cdot,k_0)} \in
\lt{[k_0,\pm\infty)\cap\Z}^2,
\;\; z\in\bbC\backslash(\dD\cup\{0\}).   \lb{2.32}
\end{equation}

The special case $t=0$ of the next result is proven in \cite[Cor.
2.16 and Thm. 2.18]{GZ06}. The general case of $t\in[0,2\pi)$ stated
below follows along the same lines and hence we omit the details for
brevity.

\begin{theorem} \lb{t2.5}
Let $t\in[0,2\pi)$ and $k_0\in\Z$. Then there exist unique
Caratheodory $($resp. anti-Caratheodory$)$ functions
$M^{(t)}_\pm(\cdot,k_0)$ such that
\begin{align}
&\binom{u^{(t)}_\pm(z,\cdot,k_0)}{v^{(t)}_\pm(z,\cdot,k_0)} =
\binom{q^{(t)}_+(z,\cdot,k_0)}{s^{(t)}_+(z,\cdot,k_0)} +
M^{(t)}_\pm(z,k_0)
\binom{p^{(t)}_+(z,\cdot,k_0)}{r^{(t)}_+(z,\cdot,k_0)} \in
\lt{[k_0,\pm\infty)\cap\Z}^2, \no \\
& \hspace*{9.5cm} z\in\C\backslash(\dD\cup\{0\}). \lb{2.37}
\end{align}
In addition, sequence
$\Big(\begin{smallmatrix}u^{(t)}_\pm(z,k,k_0)\\
v^{(t)}_\pm(z,k,k_0)\end{smallmatrix}\Big)_{k\in\Z}$ satisfies
\eqref{2.11} and is unique $($up to constant scalar multiples$)$
among all sequence that satisfy \eqref{2.11} and are square summable
near $\pm\infty$.
\end{theorem}

We will call $u^{(t)}_\pm(z,\cdot,k_0)$ and
$v^{(t)}_\pm(z,\cdot,k_0)$ Weyl--Titchmarsh solutions of $U$.
Similarly, we will call $m^{(t)}_\pm(z,k_0)$ as well as
$M^{(t)}_\pm(z,k_0)$ the half-lattice Weyl--Titchmarsh
$m$-functions associated with $U^{(t)}_{\pm,k_0}$. (See also
\cite{Si04a} for a comparison of various alternative notions of
Weyl--Titchmarsh $m$-functions for $U_{+,k_0}$.)

Using \eqref{2.26}--\eqref{2.27a}, \eqref{2.32}, and Theorem
\ref{t2.5} one also verifies that
\begin{align}
M^{(t)}_+(z,k_0) &= m^{(t)}_+(z,k_0), \quad z\in\C\backslash\dD,
\lb{2.38}
\\
M^{(t)}_+(0,k_0) &=1, \lb{2.39}
\\
M^{(t)}_-(z,k_0) &= \frac{\Re(a_{k_0}) +
i\Im(b_{k_0})m^{(t)}_-(z,k_0-1)}{i\Im(a_{k_0}) +
\Re(b_{k_0})m^{(t)}_-(z,k_0-1)} =
\f{(1-z)m^{(t)}_-(z,k_0)+(1+z)}{(1+z)m^{(t)}_-(z,k_0)+(1-z)} \no
\\
&= \f{\big(m^{(t)}_-(z,k_0)+1\big)-z\big(m^{(t)}_-(z,k_0)-1\big)}
{\big(m^{(t)}_-(z,k_0)+1\big)+z\big(m^{(t)}_-(z,k_0)-1\big)}, \quad
z\in\C\backslash\dD, \lb{2.40}
\\
M^{(t)}_-(0,k_0) &=\f{\alpha_{k_0}+e^{it}}{\alpha_{k_0}-e^{it}},
\lb{2.41}
\\
m^{(t)}_-(z,k_0) &=
\f{\Re(a_{k_0+1})-i\Im(a_{k_0+1})M^{(t)}_-(z,k_0+1)}
{\Re(b_{k_0+1})M^{(t)}_-(z,k_0+1) - i\Im(b_{k_0+1})} \no
\\
&= \f{z\big(M^{(t)}_-(z,k_0)+1\big)-\big(M^{(t)}_-(z,k_0)-1\big)}
{z\big(M^{(t)}_-(z,k_0)+1\big)+\big(M^{(t)}_-(z,k_0)-1\big)}, \quad z\in\CdD,
\lb{2.42}
\end{align}
where $a_k=1+e^{-it}\al_k$ and $b_k=1-e^{-it}\al_k$, $k\in\Z$. In
particular, one infers that $M^{(t)}_\pm$ are analytic at $z=0$.

Next, we introduce the Schur (resp. anti-Schur) functions
$\Phi^{(t)}_\pm(\cdot,k)$, $k\in\Z$, by
\begin{align}
\Phi^{(t)}_\pm(z,k) = \f{M^{(t)}_\pm(z,k)-1}{M^{(t)}_\pm(z,k)+1},
\quad z\in\C\backslash\dD. \lb{2.44}
\end{align}
Then by \eqref{2.42} and \eqref{2.44},
\begin{align}
&M^{(t)}_\pm(z,k) = \f{1+\Phi^{(t)}_\pm(z,k)}{1-\Phi^{(t)}_\pm(z,k)},
\quad m^{(t)}_-(z,k) = \f{z-\Phi^{(t)}_-(z,k)}{z+\Phi^{(t)}_-(z,k)},
\quad z\in\CdD. \lb{2.45}
\end{align}
Moreover, it follows from \eqref{2.22}, \eqref{2.44}, and Theorem
\ref{t2.5} that
\begin{align}
&\Phi^{(t)}_\pm(z,k) = \begin{cases}
ze^{it}\frac{v^{(t)}_\pm(z,k,k_0)}{u^{(t)}_\pm(z,k,k_0)}, &\text{$k$
odd,}
\\
e^{it}\frac{u^{(t)}_\pm(z,k,k_0)}{v^{(t)}_\pm(z,k,k_0)}, & \text{$k$
even,}
\end{cases} \quad k\in\Z, \; z\in\CdD, \lb{2.46}
\end{align}
where $u^{(t)}_\pm(\cdot,k,k_0)$ and $v^{(t)}_\pm(\cdot,k,k_0)$ are
the sequences defined in \eqref{2.37}. Since the Weyl--Titchmarsh
solution $\Big(\begin{smallmatrix}u^{(t)}_\pm(z,k,k_0)\\
v^{(t)}_\pm(z,k,k_0)\end{smallmatrix}\Big)_{k\in\Z}$ is unique up to
a multiplicative constant, we conclude from \eqref{2.46} that
$e^{-it}\Phi^{(t)}_\pm(\cdot,k)$ is actually $t$-independent. Thus,
fixing $t_1, t_2\in[0,2\pi)$, one computes
\begin{align}
\Phi^{(t_2)}_\pm(\cdot,k) &= e^{i(t_2-t_1)}\Phi^{(t_1)}_\pm(\cdot,k),
\\
M^{(t_2)}_\pm(\cdot,k) &=
\f{i\Im(e^{i(t_2-t_1)/2})+\Re(e^{i(t_2-t_1)/2})M^{(t_1)}_\pm(\cdot,k)}
{\Re(e^{i(t_2-t_1)/2})+i\Im(e^{i(t_2-t_1)/2})M^{(t_1)}_\pm(\cdot,k)},
\quad k\in\Z.
\end{align}

Finally, following \cite{GZ06}, we obtain the following identities,
\begin{align}
r^{(t)}_+(z,k,k_0) &= z^{k_0 \, ({\rm mod}\,2)}\ol{p^{(t)}_+(1/\ol{z},k,k_0)},
\\
s^{(t)}_+(z,k,k_0) &= -z^{k_0 \, ({\rm mod}\,2)}\ol{q^{(t)}_+(1/\ol{z},k,k_0)},
\\
r^{(t)}_-(z,k,k_0) &= -z^{(k_0 +1) \, ({\rm mod}\,2)}\ol{p^{(t)}_-(1/\ol{z},k,k_0)},
\\
s^{(t)}_-(z,k,k_0) &= z^{(k_0 +1) \, ({\rm mod}\,2)}\ol{q^{(t)}_-(1/\ol{z},k,k_0)},
\end{align}
and provide formulas for the resolvents of the half-lattice CMV operators
$U^{(t)}_{\pm,k_0}$ and the full-lattice CMV operator $U$. In the
special case $t=0$ these formulas were obtained in (2.63)--(2.66),
(2.171), (2.172), and (3.7) of \cite{GZ06}.

\begin{lemma} \lb{l2.6}
Let $t\in[0,2\pi)$, $k_0\in\Z$, and $z\in\C\bs(\dD\cup\{0\})$. Then
the resolvent $\big(U^{(t)}_{\pm,k_0}-zI\big)^{-1}$ is given in terms of its
matrix representation in the standard basis of
$\lt{[k_0,\pm\infty)\cap\Z}$ by
\begin{align}
& \big(U^{(t)}_{+,k_0}-zI\big)^{-1}(k,k') = \f{z^{- k_0 \, ({\rm mod}\,2)}}{2z} \no \\
& \hspace*{1cm} \times
\begin{cases}
p^{(t)}_+(z,k,k_0)\hatt v^{(t)}_+(z,k',k_0), & \text{$k<k'$ and $k=k'$ odd,} \\
\hatt u^{(t)}_+(z,k,k_0)r^{(t)}_+(z,k',k_0), & \text{$k>k'$ and $k=k'$ even}
\end{cases} \lb{2.85}
\\
& \quad = \f{1}{2z}
\begin{cases}
-p^{(t)}_+(z,k,k_0)\ol{\hatt u^{(t)}_+(1/\ol{z},k',k_0)}, & \text{$k<k'$ and $k=k'$ odd,} \\
\hatt u^{(t)}_+(z,k,k_0)\ol{p^{(t)}_+(1/\ol{z},k',k_0)}, & \text{$k>k'$ and $k=k'$ even,}
\end{cases} \lb{2.85a}
\\
&\hspace*{8.1cm} k,k'\in [k_0,\infty)\cap\bbZ, \no
\\
& \big(U^{(t)}_{-,k_0}-zI\big)^{-1}(k,k') = \f{z^{-(k_0 +1) \, ({\rm mod}\,2)}}{2z} \no \\
& \hspace*{1cm} \times
\begin{cases}
\hatt u^{(t)}_-(z,k,k_0)r^{(t)}_-(z,k',k_0), & \text{$k<k'$ and $k=k'$ odd,}\\
p^{(t)}_-(z,k,k_0)\hatt v^{(t)}_-(z,k',k_0), & \text{$k>k'$ and $k=k'$ even}
\end{cases} \lb{2.86}
\\
& \quad = \f{1}{2z}
\begin{cases}
-\hatt u^{(t)}_-(z,k,k_0)\ol{p^{(t)}_-(1/\ol{z},k',k_0)}, & \text{$k<k'$ and $k=k'$ odd,}\\
p^{(t)}_-(z,k,k_0)\ol{\hatt u^{(t)}_-(1/\ol{z},k',k_0)}, & \text{$k>k'$ and $k=k'$ even,}
\end{cases} \lb{2.86a}
\\
&\hspace*{7.8cm} k,k'\in (-\infty,k_0]\cap\bbZ, \no
\end{align}
where the sequences $\hatt u^{(t)}_\pm$ and $\hatt v^{(t)}_\pm$ are defined by
\begin{align}
&\binom{\hatt u^{(t)}_\pm(z,\cdot,k_0)}{\hatt v^{(t)}_\pm(z,\cdot,k_0)} =
\binom{q^{(t)}_-(z,\cdot,k_0)}{s^{(t)}_-(z,\cdot,k_0)} + m^{(t)}_\pm(z,k_0)
\binom{p^{(t)}_-(z,\cdot,k_0)}{r^{(t)}_-(z,\cdot,k_0)} \in
\ell^2([k_0,\pm\infty)\cap\Z)^2, \no \\
& \hspace*{9.8cm} z\in\bbC\backslash(\dD\cup\{0\}).
\end{align}
The corresponding result for the full-lattice resolvent of $U$ then reads
\begin{align}
& (U-zI)^{-1}(k,k') =
\frac{-z^{-k_0 \, ({\rm mod}\,2)}}{2z[M^{(t)}_+(z,k_0)-M^{(t)}_-(z,k_0)]} \no \\
& \hspace*{1cm} \times
\begin{cases}
u^{(t)}_-(z,k,k_0)v^{(t)}_+(z,k',k_0), & \text{$k<k'$ and $k=k'$ odd,} \\
u^{(t)}_+(z,k,k_0)v^{(t)}_-(z,k',k_0), & \text{$k>k'$ and $k=k'$ even}
\end{cases} \lb{2.87}
\\
& \quad = \frac{
\begin{cases}
u^{(t)}_-(z,k,k_0)\ol{u^{(t)}_+(1/\ol{z},k',k_0)}, & \text{$k<k'$ and $k=k'$ odd,} \\
u^{(t)}_+(z,k,k_0)\ol{u^{(t)}_-(1/\ol{z},k',k_0)}, & \text{$k>k'$ and $k=k'$
even,}
\end{cases}}
{2z[M^{(t)}_+(z,k_0)-M^{(t)}_-(z,k_0)]} \quad k,k' \in\Z. \lb{2.87a}
\end{align}
Moreover, since $\,U^{(t)}_{\pm,k_0}$ and $U$ are unitary and hence zero is in
the resolvent set, \eqref{2.85}--\eqref{2.87a} analytically extend to $z=0$.
\end{lemma}

\section{CMV operators with matrix-valued coefficients} \lb{s3}
In the remainder of this paper, $\Cm$ denotes the space of $m\times m$ matrices with complex-valued
entries endowed with the operator norm $\norm{\cdot}_{\Cm}$ (we use the standard Euclidean norm in $\bbC^m$). The adjoint of an element $\ga\in\Cm$ is denoted by
$\ga^*$, $I_m$ denotes the identity matrix in $\bbC^m$, and the real and imaginary
 parts of $\ga$ are defined as usual by $\Re(\ga)=(\ga+\ga^*)/2$ and $\Im(\ga) =(\ga-\ga^*)/(2i)$.

\begin{remark} \lb{r3.1}
For simplicity of exposition, we find it convenient to use the
following conventions: We denote by $\s{\Z}$ the vector space of all
$\C$-valued sequences, and by $\sm{\Z}=\s{\Z}\otimes\C^m$ the vector
space of all $\C^m$-valued sequences; that is,
\begin{align} \lb{3.1}
\phi=\{\phi(k)\}_{k\in\Z}=
\begin{pmatrix}
\vdots\\\phi(-1)\\\phi(0)\\\phi(1)\\\vdots
\end{pmatrix}\in\sm{\Z}, \quad
\phi(k)=
\begin{pmatrix}
(\phi(k))_1\\(\phi(k))_2\\\vdots\\(\phi(k))_m
\end{pmatrix}\in\C^m, {k\in\Z}.
\end{align}
Moreover, we introduce $\smn{\Z}=\sm{\Z}\otimes\C^n$, $m,n\in\N$, that is,
$\Phi=(\phi_1,\dots,\phi_n)\in\smn{\Z}$, where $\phi_j\in\sm{\Z}$
for all $j=1,\dots,n$.

We also note that $\smn{\Z}=\s{\Z}\otimes\C^{m\times n}$, $m,n\in\N$;
which is to say that the elements of $\smn{\Z}$ can be
identified with the $\C^{m\times n}$-valued sequences,
\begin{align} \lb{3.2}
\Phi=\{\Phi(k)\}_{k\in\Z}=
\begin{pmatrix}
\vdots\\\Phi(-1)\\\Phi(0)\\\Phi(1)\\\vdots
\end{pmatrix},
 \Phi(k)=
\begin{pmatrix}
(\Phi(k))_{1,1} & \hdots & (\Phi(k))_{1,n}\\
\vdots & & \vdots \\
(\Phi(k))_{m,1} & \hdots & (\Phi(k))_{m,n}
\end{pmatrix}\in\C^{m\times n}, {k\in\Z},
\end{align}
by setting $\Phi=(\phi_1,\dots,\phi_n)$, where
\begin{align} \lb{3.3}
\phi_j=
\begin{pmatrix}
\vdots\\\phi_j(-1)\\\phi_j(0)\\\phi_j(1)\\\vdots
\end{pmatrix}\in\sm{\Z}, \quad \phi_j(k)=
\begin{pmatrix}
(\Phi(k))_{1,j}\\\vdots\\(\Phi(k))_{m,j}
\end{pmatrix}\in\C^m, \; j=1,\dots,n,\; k\in\Z.
\end{align}

For the elements of $\smn{\Z}$ we define the right-multiplication by
$n\times n$ matrices with complex-valued entries by
\begin{align} \lb{3.4}
\Phi C = (\phi_1,\dots,\phi_n)
\begin{pmatrix}c_{1,1} & \dots & c_{1,n}\\
\vdots&&\vdots\\c_{n,1} & \dots & c_{n,n}\end{pmatrix} =
\left(\sum_{j=1}^n \phi_j c_{j,1},\dots,\sum_{j=1}^n \phi_j
c_{j,n}\right)\in\smn{\Z}
\end{align}
for all $\Phi\in\smn{\Z}$ and $C\in\C^{n\times n}$. In addition, for
any linear transformation $\bbA: \sm{\Z}\to\sm{\Z}$,  we define
$\bbA\Phi$ for all $\Phi=(\phi_1,\dots,\phi_n)\in\smn{\Z}$ by
\begin{align} \lb{3.5}
\bbA\Phi = (\bbA\phi_1,\dots,\bbA\phi_n)\in\smn{\Z}.
\end{align}

Given the above conventions, we note the subspace
containment: $\ltm{\Z}=\ell^2(\Z)\otimes\C^m\subset\sm{\Z}$ and
$\ltmn{\Z}=\ell^2(\Z)\otimes\C^{m\times n}\subset\smn{\Z}$. We also
note that $\ltm{\Z}$ represents a Hilbert space with scalar
product given by
\begin{align} \lb{3.6}
(\phi,\psi)_{\ltm{\Z}} = \sum_{k=-\infty}^\infty\sum_{j=1}^m
\ol{(\phi(k))_j}(\psi(k))_j, \quad \phi,\psi\in\ltm{\Z}.
\end{align}
Finally, we note that a straightforward modification of the
above definitions also yields the Hilbert space $\ltm{J}$ as well
as the sets $\ltmn{J}$, $\sm{J}$, and $\smn{J}$ for any
$J\subset\Z$.
\end{remark}

We start by introducing our basic assumption:

\begin{hypothesis} \lb{h3.2}
Let $m\in\N$ and assume $\al=\{\al_k\}_{k\in\Z}$ is a sequence of
$m \times m$ matrices with complex entries\footnote{We emphasize that $\al_k\in\Cm$, $k\in\Z$, are general (not necessarily normal) matrices.}
and such that
\begin{equation} \lb{3.7}
\norm{\al_k}_{\Cm} < 1, \quad k\in\Z.
\end{equation}
\end{hypothesis}

Given a sequence $\al$ satisfying \eqref{3.7}, we define two sequences
of positive self-adjoint $m\times m$ matrices $\{\rho_k\}_{k\in\bbZ}$
and $\{\wti\rho_k\}_{k\in\bbZ}$ by
\begin{align}
\rho_k &= [I_m-\al_k^*\al_k]^{1/2}, \quad k\in\bbZ, \lb{3.8}
\\
\wti\rho_k &= [I_m-\al_k\al_k^*]^{1/2}, \quad k\in\bbZ, \lb{3.9}
\end{align}
and two sequences of $m\times m$ matrices with positive real parts,
$\{a_k\}_{k\in\bbZ}\subset \bbC^{m\times m}$ and
$\{b_k\}_{k\in\bbZ}\subset \bbC^{m\times m}$ by
\begin{align}
a_k &= I_m+\al_k, \quad k\in\bbZ, \lb{3.10}
\\
b_k &= I_m-\al_k, \quad k \in \Z. \lb{3.11}
\end{align}
Then \eqref{3.7} implies that $\rho_k$ and $\wti\rho_k$ are
invertible matrices for all $k\in\Z$, and using elementary power series
expansions one verifies the following identities for all $k\in\Z$,
\begin{align}
&\wti\rho_k^{\pm1}\al_k = \al_k\rho_k^{\pm1}, \quad
\al_k^*\wti\rho_k^{\pm1} = \rho_k^{\pm1}\al_k^*, \lb{3.12}
 \quad
a_k^*\wti\rho_k^{-2}a_k = a_k\rho_k^{-2}a_k^*, \\
&b_k^*\wti\rho_k^{-2}b_k = b_k\rho_k^{-2}b_k^*, \quad
a_k^*\wti\rho_k^{-2}b_k + a_k\rho_k^{-2}b_k^* =
b_k^*\wti\rho_k^{-2}a_k + b_k\rho_k^{-2}a_k^* = 2I_m. \lb{3.13}
\end{align}

According to Simon \cite{Si04}, we call $\al_k$ the Verblunsky
coefficients in honor of Verblunsky's pioneering work in the theory
of orthogonal polynomials on the unit circle \cite{Ve35},
\cite{Ve36}.

Next, we introduce a sequence of $2\times 2$ block unitary matrices $\Te_k$
with $m\times m$ matrix coefficients by
\begin{equation} \lb{3.14}
\Te_k = \begin{pmatrix} -\al_k & \wti\rho_k \\ \rho_k & \al_k^*
\end{pmatrix},
\quad k \in \Z,
\end{equation}
and two unitary operators $\V$ and $\W$ on $\ltm{\Z}$ by their
matrix representations in the standard basis of $\ltm{\Z}$ by
\begin{align} \lb{3.15}
\V &= \begin{pmatrix} \ddots & & &
\raisebox{-3mm}[0mm][0mm]{\hspace*{-5mm}\Huge $0$}  \\ & \Te_{2k-2} &
& \\ & & \Te_{2k} & & \\ &
\raisebox{0mm}[0mm][0mm]{\hspace*{-10mm}\Huge $0$} & & \ddots
\end{pmatrix}, \quad
\W = \begin{pmatrix} \ddots & & &
\raisebox{-3mm}[0mm][0mm]{\hspace*{-5mm}\Huge $0$}
\\ & \Te_{2k-1} &  &  \\ &  & \Te_{2k+1} &  & \\ &
\raisebox{0mm}[0mm][0mm]{\hspace*{-10mm}\Huge $0$} & & \ddots
\end{pmatrix},
\end{align}
where
\begin{align}
\begin{pmatrix}
\V_{2k-1,2k-1} & \V_{2k-1,2k} \\ \V_{2k,2k-1}   & \V_{2k,2k}
\end{pmatrix} =  \Te_{2k},
\quad
\begin{pmatrix}
\W_{2k,2k} & \W_{2k,2k+1} \\ \W_{2k+1,2k}  & \W_{2k+1,2k+1}
\end{pmatrix} =  \Te_{2k+1},
\quad k\in\Z. \lb{3.16}
\end{align}
Moreover, we introduce the unitary operator $\U$ on
$\ltm{\Z}$ as the product of the unitary operators $\V$ and $\W$ by
\begin{equation} \lb{3.17}
\U = \V\W,
\end{equation}
or in matrix form in the standard basis of $\ltm{\Z}$, by
\begin{align}
\U &= \begin{pmatrix} \ddots &&\hspace*{-8mm}\ddots
&\hspace*{-10mm}\ddots &\hspace*{-12mm}\ddots &\hspace*{-14mm}\ddots
&&& \raisebox{-3mm}[0mm][0mm]{\hspace*{-6mm}{\Huge $0$}}
\\
&0& -\al_{0}\rho_{-1} & -\al_{0}\al_{-1}^* & -\wti\rho_{0}\al_{1} &
\wti\rho_{0}\wti\rho_{1}
\\
&& \rho_{0}\rho_{-1} &\rho_{0}\al_{-1}^* & -\al_{0}^*\al_{1} &
\al_{0}^*\wti\rho_{1} & 0
\\
&&&0& -\al_{2}\rho_{1} & -\al_{2}\al_{1}^* & -\wti\rho_{2}\al_{3} &
\wti\rho_{2}\wti\rho_{3}
\\
&&\raisebox{-4mm}[0mm][0mm]{\hspace*{-6mm}{\Huge $0$}} &&
\rho_{2}\rho_{1} & \rho_{2}\al_{1}^* & -\al_{2}^*\al_{3} &
\al_{2}^*\wti\rho_{3}&0
\\
&&&&&\hspace*{-14mm}\ddots &\hspace*{-14mm}\ddots
&\hspace*{-14mm}\ddots &\hspace*{-8mm}\ddots &\ddots
\end{pmatrix}. \lb{3.18}
\end{align}
Here terms of the form $-\al_{2k}\al_{2k-1}^*$ and
$-\al_{2k}^*\al_{2k+1}$, $k\in\Z$, represent the diagonal entries
$\U_{2k-1,2k-1}$ and $\U_{2k,2k}$ of the infinite matrix $\U$ in
\eqref{3.18}, respectively.  Then, with $\deven$ and $\dodd$ defined
in \eqref{2.10}, and by analogy with \eqref{2.9}, we see that as an operator acting
upon $\ltm{\Z}$, $\U$ can be represented
by
\begin{align}
\U&= \rho\rho^-\deven S^{--} + [\rho(\al^-)^*\deven -\al^*\rho\ \dodd ]S^-
- \al^*\al^+\deven - \al^+\al^*\dodd \notag\\
& \hspace{4mm}+ [\al^*\wti\rho^{\hspace{2pt}+}\deven - \wti\rho^{\hspace{2pt}+}\al^{++}\dodd ]S^+
+\wti\rho^{\hspace{2pt}+}\wti\rho^{\hspace{2pt}++}\dodd S^{++}.
\end{align}
We continue to call the operator $\U$ on
$\ltm{\Z}$ the CMV operator since \eqref{3.14}--\eqref{3.18} in the
context of the scalar-valued semi-infinite (i.e., half-lattice) case
were obtained by Cantero, Moral, and Vel\'azquez in \cite{CMV03} in
2003.  Then, in analogy with Lemma~\ref{l2.2}, the following result is proven in
\cite{CGZ07}:

\begin{lemma} \lb{l3.3}
Let $z\in\bbC\backslash\{0\}$ and $\{U(z,k)\}_{k\in\bbZ},
\{V(z,k)\}_{k\in\bbZ}$ be two $\Cm$-valued sequences. Then the
following items $(i)$--$(iii)$ are equivalent:
\begin{align}
(i) & \quad  (\U U(z,\cdot))(k) = z U(z,k), \quad (\W U(z,\cdot))(k)=z
V(z,k), \quad k\in\Z.
\\
(ii) & \quad (\W U(z,\cdot))(k) = z V(z,k), \quad (\V V(z,\cdot))(k) =
U(z,k), \quad k\in\Z. \lb{3.21}
\\
(iii) & \quad \binom{U(z,k)}{V(z,k)} = \T(z,k)
\binom{U(z,k-1)}{V(z,k-1)}, \quad k\in\Z.\lb{3.22}
\end{align}
Here $\U$, $\V$, and $\W$  are understood in the sense of difference
expressions on $\smm{\Z}$ rather than difference operators on
$\ltm{\Z}$ $($cf.\ Remark \ref{r3.1}$)$ and the transfer matrices
$\T(z,k)$, $z\in\Cz$, $k\in\Z$, are defined by
\begin{equation}\lb{3.23}
\T(z,k) = \begin{cases}
\begin{pmatrix}
\wti\rho_{k}^{-1}\al_{k} & z\wti\rho_{k}^{-1} \\
z^{-1}\rho_{k}^{-1} & \rho_{k}^{-1}\al_{k}^*
\end{pmatrix},  & \text{$k$ odd,}
\\[20pt]
\begin{pmatrix}
\rho_{k}^{-1}\al_{k}^* & \rho_{k}^{-1} \\
\wti\rho_{k}^{-1} & \wti\rho_{k}^{-1}\al_{k}
\end{pmatrix}, & \text{$k$ even.}
\end{cases}
\end{equation}
\end{lemma}

\begin{definition}\lb{d3.4}
If for some reference point, $k_0\in\Z$,  we  modify Hypothesis~\ref{h3.2} by allowing $\al_{k_0}=\ga$, where $\ga\in\C^{m\times m}$ is unitary, then the resulting operators defined in \eqref{3.15}--\eqref{3.17} will be denoted by
$\V_{k_0}^{(\ga)}, \ \W_{k_0}^{(\ga)}$ and $\U_{k_0}^{(\ga)}$.
\end{definition}

\begin{remark}
Strictly, speaking,  allowing  $\al_{k_0}$ to be unitary for some reference point $k_0\in\Z$ contradicts our basic Hypothesis \ref{h3.2}.
However, as long as the exception to Hypothesis \ref{h3.2} refers to
only one site, we will safely ignore this inconsistency in favor of
the notational simplicity it provides by avoiding the introduction of
a properly modified hypothesis on $\al=\{\al_k\}_{k\in\bbZ}$ and will refer to $\U_{k_0}^{(\ga)}$ as a CMV operator.
\end{remark}

The operator $\U_{k_0}^{(\ga)}$ splits into a direct sum of two half-lattice operators $\U_{-,k_0-1}^{(\ga)}$ and $\U_{-,k_0}^{(\ga)}$ acting on $\ltm{(-\infty,k_0-1]\cap\Z}$ and on $\ltm{[k_0,\infty)\cap\Z}$,
respectively.  Explicitly, one obtains
\begin{align}
\U_{k_0}^{(\ga)} = \U_{-,k_0-1}^{(\ga)} \oplus \U_{+,k_0}^{(\ga)} \, \text{ in } \,
\ltm{(-\infty,k_0-1]\cap\Z} \oplus \ltm{[k_0,\infty)\cap\Z}.
\end{align}
 Similarly,
one obtains $\W_{-,k_0-1}^{(\ga)}$, $\V_{-,k_0-1}^{(\ga)}$ and $\W_{+,k_0}^{(\ga)}$,
$\V_{+,k_0}^{(\ga)}$ such that
\begin{equation}\lb{3.25}
\V_{k_0}^{(\ga)}=\V_{-,k_0-1}^{(\ga)}\oplus \V_{+,k_0}^{(\ga)},
\quad \W_{k_0}^{(\ga)}=\W_{-,k_0-1}^{(\ga)}\oplus \W_{+,k_0}^{(\ga)},
\end{equation}
and hence
\begin{equation}\lb{3.26}
\U_{\pm,k_0}^{(\ga)} = \V_{\pm,k_0}^{(\ga)} \W_{\pm,k_0}^{(\ga)}, \quad
\U_{k_0}^{(\ga)} = \V_{k_0}^{(\ga)} \W_{k_0}^{(\ga)}.
\end{equation}
For the special case when $\ga=I_m$, we simplify our notation by writing
\begin{align}
&\U_{k_0}=\U_{k_0}^{(\ga= I_m)}=\V_{k_0}^{(\ga= I_m)}\W_{k_0}^{(\ga= I_m)}=
\V_{k_0}\W_{k_0}\lb{3.27}\\
&\U_{\pm,k_0}=\U_{\pm,k_0}^{(\ga= I_m)}=\V_{\pm,k_0}^{(\ga= I_m)}\W_{\pm,k_0}^{(\ga= I_m)}=
\V_{\pm,k_0}\W_{\pm,k_0}
\end{align}

In analogy with the scalar case, when emphasizing dependence of these operators on the sequence $\al=\{\al_k\}_{k\in\Z}$, we write, for example, $\U_{\al}$ and $\U_{\al;k_0}^{(\ga)}$. Also in analogy to the scalar case, let  $\si, \ta\in\Cm$ be unitary, and let $\A$ and $\wti\A$ be the unitary operators defined on $\ltm{\Z}$ by
\begin{equation}\lb{3.29}
\A=\si\dodd + \ta\deven, \quad \wti\A=\ta\dodd + \si\deven,
\end{equation}
Then, for the full-lattice and direct sum of half-lattice CMV operators, we obtain the following analogs of \eqref{2.17b} and \eqref{2.17c}:
\begin{align}
& \A\U_{\be}\A^* = \big[\A\V_{\be}\wti\A^*\big] \big[\wti\A\W_{\be}\A^*\big]
= \V_{\al}\W_{\al} = \U_{\al},\lb{3.30}
\\
& \A\U_{\be;k_0}^{(\ga)}\A^* = \big[\A\V_{\be;k_0}^{(\ga)}\wti\A^*\big]
\big[\wti\A\W_{\be;k_0}^{(\ga)}\A^*\big] =
\V_{\al;k_0}^{(\si\ga\ta^*)}\W_{\al;k_0}^{(\si\ga\ta^*)} =
\U_{\al;k_0}^{(\si\ga\ta^*)}.\lb{3.31}
\end{align}
where $\be=\{\be_k\}_{k\in\Z}$, and $\al_k=\si\be_k\ta^*$.
In particular, when $\si=\ga^{-1/2}$, $\ta=\ga^{1/2}$, we note, by \eqref{3.27} and \eqref{3.31}, that
\begin{equation}\lb{3.32}
\A\U_{\be;k_0}^{(\ga)}\A^*=\V_{\al;k_0}\W_{\al;k_0}=\U_{\al;k_0}.
\end{equation}

The unitary transformations cited in \eqref{3.29}--\eqref{3.31} are relevant to our next result because of the following observation about $n\times n$ complex matrices:
Let $\al\in\Cm$, where we are not assuming that $\al$ is unitary. Then, $\al$ has a (not necessarily unique) polar decomposition, $\al=U|\al|$, where $U, |\al|\in\Cm$, $U$ is unitary, and $|\al|=(\al^*\al)^{1/2}\geq 0$ is nonnegative (cf., e.g.,
\cite[Theorem\ 3.1.9 (c)]{HJ94}). The nonnegative  matrix $|\al|$ can then be diagonalized: $|\al|=U^*_0DU_0$, where $U_0, D\in\Cm$, $U_0$ is unitary, and $D$ is diagonal. Thus, each $\al\in\Cm$ has a (not necessarily unique) factorization of the form
\begin{equation}\lb{3.33}
\al = \si\be\ta,
\end{equation}
where, $\si,\be,\ta\in\Cm$, where $\si$ and $\ta$ are unitary, and where $\be$ is diagonal.

We now present our principal result on
rank~$m$ perturbations:
\begin{theorem}\lb{t3.6}
Given $\U_{\al}$, fix $k_0\in\Z$ and let
$\al_{k_0}=\si_{k_0}\be_{k_0}\ta_{k_0}^*$ be a factorization
for $\al_{k_0}$, as described in \eqref{3.33}, where $\si_{k_0},\
\ta_{k_0}\in\Cm$ are unitary, and where
$\be_{k_0}\in\Cm$ is the diagonal matrix
$\be_{k_0}=\diag[\be_{k_0,1},\dots,\be_{k_0,m}].$  Let
$\ga_j=\si_{k_0}\te_j\ta_{k_0}^*$, $j=1,2$, where
$\te_1=\diag[e^{it_1},\dots,e^{it_m}]$, and
$\te_2=\diag[e^{is_1},\dots,e^{is_m}]$. Let
$\U_{\al;k_0}^{(\ga_1,\ga_2)}$ denote the  unitary operator acting
on $\ltm{\Z}$ and defined by
\begin{equation}
\U_{\al;k_0}^{(\ga_1,\ga_2)}=\U_{\al;-,k_0-1}^{(\ga_1)}\oplus\U_{\al;+,k_0}^{(\ga_2)}.
\end{equation}
Then, for an arbitrary unitary matrix $\ga\in\Cm$, the difference $\U_{\al}-\U_{\al;k_0}^{(\ga)}$ has rank greater than $m$, while the differences    $\U_{\al}-\U_{\al;k_0}^{(\ga_1,\ga_2)}$ and $(\U_{\al}-zI)^{-1} - \big(\U_{\al;k_0}^{(\ga_1,\ga_2)} -zI\big)^{-1}$ are of rank~$m$ if and only if \
$t_j=2\arg [i(\be_{k_0,j}e^{-is_j/2} - e^{is_j/2})]$, $j=1,\dots,m$, and otherwise possess rank greater than $m$.
\end{theorem}
\begin{proof}
 $\U^{(\ga)}_{\al;k_0}=\V^{(\ga)}_{\al;k_0}\W^{(\ga)}_{\al;k_0}$ by
\eqref{3.26}, and hence, it follows from \eqref{3.15} that
\begin{align}
\U_{\al}-\U^{(\ga)}_{\al;k_0} =
\begin{cases}
\V_{\al}\big(\W_{\al}-\W^{(\ga)}_{\al;k_0}\big), & \text{$k_0$ odd,} \\
\big(\V_{\al}-\V^{(\ga)}_{\al;k_0}\big)\W_{\al}, & \text{$k_0$ even.}
\end{cases}
\end{align}
For $k_0$ odd, $D=\W_{\al}-\W^{(\ga)}_{\al;k_0}$ is block-diagonal with all
of its $2m\times 2m$ blocks on the diagonal being zero except for one
which has the following form
\begin{equation}
A=
\begin{pmatrix}
D_{k_0-1,k_0-1} & D_{k_0-1,k_0} \\ D_{k_0,k_0-1}  & D_{k_0,k_0}
\end{pmatrix}
=
\begin{pmatrix}
-\al_{k_0}+\ga & \wti\rho_{k_0} \\ \rho_{k_0} & \al_{k_0}^*-\ga^*
\end{pmatrix}
\end{equation}

Note that the following subspaces have only trivial intersection
with $\ker(A)$ since $\rho_{k_0}$ and $\wti\rho_{k_0}$ are
invertible matrices,
\begin{equation}
S_1=\left\{ \begin{pmatrix}\xi\\0 \end{pmatrix}\in\C^{2m}\mid \xi\in\C^{m}\right\},\quad
S_2=\left\{ \begin{pmatrix}0\\ \ea \end{pmatrix}\in\C^{2m}\mid \ea\in\C^{m}\right\}.
\end{equation}
As a consequence, $\rank(A)\ge m$.  Further note, for any $c\in\C$, that
\begin{equation}\lb{3.38}
\begin{pmatrix}\xi \\ c\xi \end{pmatrix} \in \ker(A)
\end{equation}
only when $\xi=0\in\C^m$.  To see this, assume that \eqref{3.38}
holds for some $c\in\C$ and some $\xi\ne 0\in\C^m$. It follows that
\begin{equation}\lb{3.39}
\begin{pmatrix} 0\\0\end{pmatrix}
= \begin{pmatrix} c\xi\\ \xi \end{pmatrix}^*A \begin{pmatrix} \xi\\ c\xi \end{pmatrix}=
\begin{pmatrix}-\bar c \xi^*(\al_{k_0}-\ga)\xi+ |c|^2 \xi^*\wti\rho_{k_0}\xi\\
\xi^*\rho\xi + c\xi^*(\al_{k_0}-\ga)^*\xi \end{pmatrix}
\end{equation}
and hence by conjugation in the first line of \eqref{3.39} that
\begin{align}
0&=-c \xi^*(\al_{k_0}-\ga)^*\xi+ |c|^2 \xi^*\wti\rho_{k_0}\xi,\\
0&=\xi^*\rho\xi + c\xi^*(\al_{k_0}-\ga)^*\xi.
\end{align}
Summing these equations, we see that
$\xi^*(\rho_{k_0}+|c|^2\wti\rho_{k_0})\xi=0$. However, strict
positivity of the self-adjoint matrix
$\rho_{k_0}+|c|^2\wti\rho_{k_0}$ implies that $\xi=0\in\C^m$; a
contradiction.

Noting again that $\rank(A)\ge m$, assume that $\rank(A)=m$. Given that
\begin{equation}
\C^{2m} = S_1\oplus S_2, \quad \ker(A)\cap S_j =\{0\}, j=1,2,
\end{equation}
then there exists a matrix $M\in\Cm$ such that
\begin{equation}
\ker(A) = \left\{\begin{pmatrix}\xi\\M\xi \end{pmatrix} \mid \xi\in\C^m  \right\}.
\end{equation}
However, this implies the existence of some $\xi\ne 0\in\C^m$ and some $c\in\C$ such that \eqref{3.38} holds; thus, again a contradiction. Hence, the rank of $A$, and as a consequence the rank of $\U_{\al}-\U_{\al;k_0}^{(\ga)}$, are greater than $m$. The proof when $k_0\in\Z$ is even is similar to that just given.

Now consider the rank of the difference $\U_{\al}-\U_{\al;k_0}^{(\ga_1,\ga_2)}$, first noting by \eqref{3.25} and \eqref{3.26} that $\U^{(\ga_1,\ga_2)}_{\al;k_0}=\V^{(\ga_1,\ga_2)}_{\al;k_0}\W^{(\ga_1,\ga_2)}_{\al;k_0}$ and hence, by \eqref{3.15} that
\begin{align}
\U_{\al}-\U^{(\ga_1,\ga_2)}_{\al;k_0} =
\begin{cases}
\V_{\al}\big(\W_{\al}-\W^{(\ga_1,\ga_2)}_{\al;k_0}\big), & \text{$k_0$ odd,} \\
\big(\V_{\al}-\V^{(\ga_1,\ga_2)}_{\al;k_0}\big)\W_{\al}, & \text{$k_0$ even.}
\end{cases}
\end{align}
We again proceed by assuming that $k_0\in\Z$ is odd, and letting
$D=\W_{\al}-\W^{(\ga_1,\ga_2)}_{\al;k_0}$ be the block-diagonal matrix all of whose
$2m\times 2m$ blocks on the diagonal are zero except for one
which has the following form:
\begin{align}
\begin{pmatrix}
D_{k_0-1,k_0-1} & D_{k_0-1,k_0} \\ D_{k_0,k_0-1}  & D_{k_0,k_0}
\end{pmatrix}
&=\begin{pmatrix}-\al_{k_0} & \wti\rho_{k_0} \\ \rho_{k_0} & \al_{k_0}^*\end{pmatrix}
-
\begin{pmatrix}-\si_{k_0}\te_1\ta_{k_0}^* & 0 \\ 0 & \ta_{k_0}\te_2^*\si_{k_0}^*\end{pmatrix}\\
&=
\begin{pmatrix}\si_{k_0} & 0 \\ 0 & \ta_{k_0}\end{pmatrix}
\begin{pmatrix}-\be_{k_0}+\te_1 & \ka_{k_0} \\ \ka_{k_0} & \be_{k_0}^*-\te_2^*\end{pmatrix}
\begin{pmatrix}\ta_{k_0}^* & 0 \\ 0 & \si_{k_0}^*\end{pmatrix},
\end{align}
where $\ka_{k_0} = (I_m-\be_{k_0}\be_{k_0}^*)^{1/2} = (I_m-\be_{k_0}^*\be_{k_0})^{1/2}=\diag[\ka_{k_0,1},\dots,\ka_{k_0,m}]$. Then, the rank of the difference
$D=\W_{\al}-\W^{(\ga_1,\ga_2)}_{\al;k_0}$ equals the rank of the matrix
\begin{equation}\lb{3.47}
B=\begin{pmatrix}-\be_{k_0}+\te_1 & \ka_{k_0} \\ \ka_{k_0} & \be_{k_0}^*-\te_2^*\end{pmatrix}.
\end{equation}

Since all $m\times m$ matrices on the right-hand side of
\eqref{3.47} are diagonal, performing an equivalent set of
permutations on the rows and on the columns of the matrix $B$ in
\eqref{3.47} results in a block diagonal matrix with $2\times 2$
blocks arrayed along the diagonal, each of the form,
\begin{equation}
B_j=
\begin{pmatrix}
-\be_{k_0,j}+e^{it_j}& \ka_{k_0,j}\\  \ka_{k_0,j}&\ol\be_{k_0,j}-e^{-is_j}
\end{pmatrix},  \quad j=1,\dots,m.
\end{equation}
It follows that $\rank(B)=m$ precisely when $\rank(B_j) =1$ for all $j=1,\dots,m$, and otherwise, that $\rank(B)>m$. We now observe that each $B_j$ has the form given in \eqref{2.21}. Hence, by the derivation following from \eqref{2.22a} we see that $\det(B_j)=0$, and hence that $\rank(B_j)=1$, precisely when
$
t_j=2\arg [i(\be_{k_0,j}e^{-is_j/2} - e^{is_j/2})].$

As in the scalar case discussed in Theorem~\ref{t2.3}, the statement
in this theorem for the resolvents follows from the result just
proven for the difference $\U_{\al}-\U_{\al;k_0}^{(\ga_1,\ga_2)}$
and from the identity,
\begin{equation}
(\U_{\al}-zI)^{-1} - \big(\U_{\al;k_0}^{(\ga_1,\ga_2)} -zI\big)^{-1}
=
-(\U_{\al}-zI)^{-1}
\big[\U_{\al}-\U^{(\ga_1,\ga_2)}_{\al;k_0}\big] \big(\U^{(\ga_1,\ga_2)}_{\al;k_0}-zI\big)^{-1}.
\end{equation}
\end{proof}

\begin{remark}
In particular, we note that the difference $\big[\U_{\al}-\U_{\al;k_0}^{(\ga)}\big]$ has rank greater than $m$ when $\ga = I_m$.
\end{remark}

Next we present formulas for
different unitary  $\ga\in\Cm$, linking various spectral theoretic objects
associated with half-lattice CMV operators $\U^{(\ga)}_{\pm,k_0}$. For the special case when $\ga=I_m$, these objects, and the relationships described below, have proven exceptionally useful (see, e.g., \cite{CGZ07}, \cite{CGZ08}, \cite{GZ06}, \cite{GZ06a}, and \cite{Zi08}).

We begin with an analog of Lemma
\ref{l3.3} for difference expressions $\U^{(\ga)}_{\pm,k_0}$,
$\V^{(\ga)}_{\pm,k_0}$, and $\W^{(\ga)}_{\pm,k_0}$. For the special case
$\ga=I_m$, the result below is proven in \cite[Lemma 2.4]{CGZ07}. The general case
below follows immediately from the special case and the observation
of unitary equivalence in \eqref{3.32}.

\begin{lemma} \lb{l3.8}
Let $z\in\C\backslash\{0\}$, $k_0\in\Z$, and let $\ga\in\Cm$ be unitary.
Let
$\big\{\hatt P^{(\ga)}_+(z,k,k_0)\big\}_{k\geq k_0}$,
$\big\{\hatt R^{(\ga)}_+(z,k,k_0)\big\}_{k\geq k_0}$ be two $\Cm$-valued sequences. Then the
following items $(i)$--$(iii)$ are equivalent:
\begin{align}
(i) &\quad \big(\U^{(\ga)}_{+,k_0} \hatt P^{(\ga)}_+(z,\cdot,k_0)\big)(k) = z \hatt
P^{(\ga)}_+(z,k,k_0), \no \\
&\quad \big(\W^{(\ga)}_{+,k_0} \hatt P^{(\ga)}_+(z,\cdot,k_0)\big)(k) = z \hatt
R^{(\ga)}_+(z,k,k_0), \quad k\geq k_0.
\\
(ii) &\quad \big(\W^{(\ga)}_{+,k_0} \hatt P^{(\ga)}_+(z,\cdot,k_0)\big)(k) = z
\hatt R^{(\ga)}_+(z,k,k_0), \no \\
&\quad \big(\V^{(\ga)}_{+,k_0} \hatt R^{(\ga)}_+(z,\cdot,k_0)\big)(k) = \hatt
P^{(\ga)}_+(z,k,k_0), \quad k\geq k_0.
\\
(iii) &\quad \binom{\hatt P^{(\ga)}_+(z,k,k_0)}{\hatt R^{(\ga)}_+(z,k,k_0)}
= \T(z,k) \binom{\hatt P^{(\ga)}_+(z,k-1,k_0)}{\hatt
R^{(\ga)}_+(z,k-1,k_0)}, \quad k > k_0, \no \\
&\quad \hatt P^{(\ga)}_+(z,k_0,k_0) =
\begin{cases}
z\ga \hatt R^{(\ga)}_+(z,k_0,k_0), & \text{$k_0$ odd}, \\
\ga^* \hatt R^{(\ga)}_+(z,k_0,k_0), & \text{$k_0$ even}.
\end{cases}
\end{align}
Similarly, let $\big\{\hatt P^{(\ga)}_-(z,k,k_0)\big\}_{k\leq k_0}$,
$\big\{\hatt R^{(\ga)}_-(z,k,k_0)\big\}_{k\leq k_0}$ be two $\Cm$-valued sequences. Then the
following items $(iv)$--$(vi)$are equivalent:
\begin{align}
(iv) &\quad \big(\U^{(\ga)}_{-,k_0} \hatt P^{(\ga)}_-(z,\cdot,k_0)\big)(k) = z
\hatt P^{(\ga)}_-(z,k,k_0), \no \\
&\quad \big(\W^{(\ga)}_{-,k_0} \hatt P^{(\ga)}_-(z,\cdot,k_0)\big)(k) = z \hatt
R^{(\ga)}_-(z,k,k_0), \quad k\leq k_0.
\\
(v) &\quad \big(\W^{(\ga)}_{-,k_0} \hatt P^{(\ga)}_-(z,\cdot,k_0)\big)(k) = z
\hatt R^{(\ga)}_-(z,k,k_0), \no \\
&\quad \big(\V^{(\ga)}_{-,k_0} \hatt R^{(\ga)}_-(z,\cdot,k_0)\big)(k) = \hatt
P^{(\ga)}_-(z,k,k_0), \quad k\leq k_0.
\\
(vi) &\quad \binom{\hatt P^{(\ga)}_-(z,k-1,k_0)}{\hatt
R^{(\ga)}_-(z,k-1,k_0)}=\T(z,k)^{-1} \binom{\hatt
P^{(\ga)}_-(z,k,k_0)}{\hatt R^{(\ga)}_-(z,k,k_0)}, \quad k \leq k_0,\no
\\
&\quad \hatt P^{(\ga)}_-(z,k_0,k_0) =
\begin{cases}
-\ga\hatt R^{(\ga)}_-(z,k_0,k_0), & \text{$k_0$ odd,}
\\
-z\ga^* \hatt R^{(\ga)}_-(z,k_0,k_0), & \text{$k_0$ even.}
\end{cases}
\end{align}
\end{lemma}

Next, we denote by
$\Big(\begin{smallmatrix}P^{(\ga)}_\pm(z,k,k_0) \\
R^{(\ga)}_\pm(z,k,k_0)\end{smallmatrix}\Big)_{k\in\Z}$ and
$\Big(\begin{smallmatrix} Q^{(\ga)}_\pm(z,k,k_0) \\ S^{(\ga)}_\pm(z,k,k_0)
\end{smallmatrix}\Big)_{k\in\Z}$,
$z\in\bbC\backslash\{0\}$, four linearly independent solutions of
\eqref{3.21} satisfying the following initial conditions:
\begin{align}
&\binom{P^{(\ga)}_+(z,k_0,k_0)}{R^{(\ga)}_+(z,k_0,k_0)} =
\begin{cases}
\binom{z\ga^{1/2}}{\ga^{-1/2}}, & \text{$k_0$ odd,} \\[1mm]
\binom{\ga^{-1/2}}{\ga^{1/2}}, & \text{$k_0$ even,}
\end{cases} \; \; \,
\binom{Q^{(\ga)}_+(z,k_0,k_0)}{S^{(\ga)}_+(z,k_0,k_0)} =
\begin{cases}
\binom{z\ga^{1/2}}{-\ga^{-1/2}}, & \text{$k_0$ odd,} \\[1mm]
\binom{-\ga^{-1/2}}{\ga^{1/2}}, & \text{$k_0$ even.}
\end{cases} \lb{3.56}
\\
&\binom{P^{(\ga)}_-(z,k_0,k_0)}{R^{(\ga)}_-(z,k_0,k_0)} =
\begin{cases}
\binom{\ga^{1/2}}{-\ga^{-1/2}}, & \text{$k_0$ odd,} \\[1mm]
\binom{-z\ga^{-1/2}}{\ga^{1/2}}, & \text{$k_0$ even,}
\end{cases}
\binom{Q^{(\ga)}_-(z,k_0,k_0)}{S^{(\ga)}_-(z,k_0,k_0)} =
\begin{cases}
\binom{\ga^{1/2}}{\ga^{-1/2}}, & \text{$k_0$ odd, }\\[1mm]
\binom{z\ga^{-1/2}}{\ga^{1/2}}, & \text{$k_0$ even.}
\end{cases} \lb{3.57}
\end{align}
Then, $P^{(\ga)}_\pm(z,k,k_0)$, $Q{(\ga)}_\pm(z,k,k_0)$,
$R{(\ga)}_\pm(z,k,k_0)$, and $S{(\ga)}_\pm(z,k,k_0)$, $k,k_0\in\Z$, are
$\Cm$-valued Laurent polynomials in $z$.

\begin{lemma} \lb{l3.9}
Let $z\in\C\backslash\{0\}$, $k_0\in\Z$, and let $\ga_j\in\Cm$, $j=1,2$, be unitary.
With
$C_1=\frac12(\ga_1^{-1/2}\ga_2^{1/2} + \ga_1^{1/2}\ga_2^{-1/2})$ and 
$D_1=\frac12(\ga_1^{-1/2}\ga_2^{1/2} - \ga_1^{1/2}\ga_2^{-1/2})$, 
then,
\begin{align}
&\binom{Q^{(\ga_2)}_\pm(z,\cdot,k_0)}{S^{(\ga_2)}_\pm(z,\cdot,k_0)} =
\binom{Q^{(\ga_1)}_\pm(z,\cdot,k_0)}{S^{(\ga_1)}_\pm(z,\cdot,k_0)}C_1 +
\binom{P^{(\ga_1)}_\pm(z,\cdot,k_0)}{R^{(\ga_1)}_\pm(z,\cdot,k_0)}D_1\lb{3.58}
\\
&\binom{P^{(\ga_2)}_\pm(z,\cdot,k_0)}{R^{(\ga_2)}_\pm(z,\cdot,k_0)} =
\binom{Q^{(\ga_1)}_\pm(z,\cdot,k_0)}{S^{(\ga_1)}_\pm(z,\cdot,k_0)}D_1 +
\binom{P^{(\ga_1)}_\pm(z,\cdot,k_0)}{R^{(\ga_1)}_\pm(z,\cdot,k_0)}C_1.
\end{align}
With
$C_2=(2z^{k_0 \, ({\rm mod}\,2)})^{-1}(\ga_1^{-1/2}\ga_2^{1/2} - z\ga_1^{1/2}\ga_2^{-1/2})$, 
$D_2=(2z^{k_0 \, ({\rm mod}\,2)})^{-1}(\ga_1^{-1/2}\ga_2^{1/2} + z\ga_1^{1/2}\ga_2^{-1/2}),$ then
\begin{align}
&\binom{Q^{(\ga_2)}_-(z,\cdot,k_0)}{S^{(\ga_2)}_-(z,\cdot,k_0)} =
\binom{Q^{(\ga_1)}_+(z,\cdot,k_0)}{S^{(\ga_1)}_+(z,\cdot,k_0)}C_2 +
\binom{P^{(\ga_1)}_+(z,\cdot,k_0)}{R^{(\ga_1)}_+(z,\cdot,k_0)}D_2,
\\
&\binom{P^{(\ga_2)}_-(z,\cdot,k_0)}{R^{(\ga_2)}_-(z,\cdot,k_0)} =
\binom{Q^{(\ga_1)}_+(z,\cdot,k_0)}{S^{(\ga_1)}_+(z,\cdot,k_0)}D_2 +
\binom{P^{(\ga_1)}_+(z,\cdot,k_0)}{R^{(\ga_1)}_+(z,\cdot,k_0)}C_2.\lb{3.61}
\end{align}
With $C_3=\frac12(\ga_1^{-1/2}\wti\rho_{k_0}^{-1}\al_{k_0}\ga_2^{-1/2}-
\ga_1^{1/2}\rho_{k_0}^{-1}\al_{k_0}^*\ga_2^{1/2}) +
\frac12(\ga_1^{-1/2}\wti\rho_{k_0}^{-1}\ga_2^{1/2}-
\ga_1^{1/2}\rho_{k_0}^{-1}\ga_2^{-1/2})$ and 
$D_3=\frac12(\ga_1^{-1/2}\wti\rho_{k_0}^{-1}\al_{k_0}\ga_2^{-1/2}+
\ga_1^{1/2}\rho_{k_0}^{-1}\al_{k_0}^*\ga_2^{1/2}) +
\frac12(\ga_1^{-1/2}\wti\rho_{k_0}^{-1}\ga_2^{1/2}+
\ga_1^{1/2}\rho_{k_0}^{-1}\ga_2^{-1/2}),$
then
\begin{equation}
\binom{Q^{(\ga_2)}_-(z,\cdot,k_0-1)}{S^{(\ga_2)}_-(z,\cdot,k_0-1)} =
\binom{Q^{(\ga_1)}_+(z,\cdot,k_0)}{S^{(\ga_1)}_+(z,\cdot,k_0)}C_3+
\binom{P^{(\ga_1)}_+(z,\cdot,k_0)}{R^{(\ga_1)}_+(z,\cdot,k_0)}D_3.\lb{3.62}
\end{equation}
With $C_4 =-\frac12(\ga_1^{-1/2}\wti\rho_{k_0}^{-1}\al_{k_0}\ga_2^{-1/2}+
\ga_1^{1/2}\rho_{k_0}^{-1}\al_{k_0}^*\ga_2^{1/2}) +
\frac12(\ga_1^{-1/2}\wti\rho_{k_0}^{-1}\ga_2^{1/2}+
\ga_1^{1/2}\rho_{k_0}^{-1}\ga_2^{-1/2})$ and 
$D_4=-\frac12(\ga_1^{-1/2}\wti\rho_{k_0}^{-1}\al_{k_0}\ga_2^{-1/2}-
\ga_1^{1/2}\rho_{k_0}^{-1}\al_{k_0}^*\ga_2^{1/2}) +
\frac12(\ga_1^{-1/2}\wti\rho_{k_0}^{-1}\ga_2^{1/2}-
\ga_1^{1/2}\rho_{k_0}^{-1}\ga_2^{-1/2}),$
\begin{equation}
\binom{P^{(\ga_2)}_-(z,\cdot,k_0-1)}{R^{(\ga_2)}_-(z,\cdot,k_0-1)} =
\binom{Q^{(\ga_1)}_+(z,\cdot,k_0)}{S^{(\ga_1)}_+(z,\cdot,k_0)}C_4 +
\binom{P^{(\ga_1)}_+(z,\cdot,k_0)}{R^{(\ga_1)}_+(z,\cdot,k_0)}D_4.\lb{3.63}
\end{equation}
\end{lemma}
\begin{proof}
Since each  sequence in equations \eqref{3.58}--\eqref{3.63}  satisfies the recurrence relation \eqref{3.22}, it is necessary only to check equality at $k=k_0$. Substituting \eqref{3.56}, \eqref{3.57} into \eqref{3.58}--\eqref{3.61} suffices to verify the identities. In addition to use of \eqref{3.56}, \eqref{3.57}, application of the transfer matrix $\T(z,k_0)$ in \eqref{3.23} to the left sides of \eqref{3.62}, \eqref{3.63} suffices in the verification of the later two identities.
\end{proof}

Next, we introduce half-lattice Weyl-Titchmarsh m-functions associated with the half-lattice CMV operators, $\U_{\pm,k_0}^{(\ga)}$, described in \eqref{3.26}.

Let $\De_k=\{\De_k(\ell)\}_{\ell\in\Z}\in\smm{\Z}$, $k\in\Z$, denote
the sequences of $m\times m$ matrices defined by
\begin{align}
(\De_k)(\ell) =
\begin{cases}
I_m, & \ell=k,\\
0, & \ell\neq k,
\end{cases}
\quad k,\ell\in\Z.
\end{align}
Then using right-multiplication by $m\times m$ matrices on
$\smm{\Z}$ defined in Remark \ref{r3.1}, we get the
identity
\begin{align}
(\De_k X)(\ell) =
\begin{cases}
X, & \ell=k,\\
0, & \ell\neq k,
\end{cases}  \quad  X\in\Cm,
\end{align}
and hence consider $\De_k$ as a map $\De_k\colon \Cm\to\smm{\Z}$. In
addition, we introduce the map $\De_k^*\colon \smm{\Z}\to\Cm$,
$k\in\Z$, defined by
\begin{align}
\De_k^* \Phi = \Phi(k), \,\text{ where }\,
\Phi=\{\Phi(k)\}_{k\in\Z}\in\smm{\Z}.
\end{align}
Similarly, one introduces the corresponding maps with $\bbZ$ replaced by
$[k_0,\pm\infty)\cap\Z$, $k_0\in\Z$, which, for notational brevity, we
will also denote by $\De_k$ and $\De_k^*$, respectively.

For $k_0\in\Z$, and for a unitary $\ga\in\Cm$, let $d\Om_\pm^{(\ga)}(\cdot,k_0)$ denote the $\Cm$-valued measures  on $\dD$
defined by
\begin{equation}
d\Om_\pm^{(\ga)}(\zeta,k_0) = d(\De_{k_0}^* E_{\U_{\pm,k_0}^{(\ga)}}(\zeta)\De_{k_0}),
\quad \zeta\in\dD,
\end{equation}
where $E_{\U_{\pm,k_0}^{(\ga)}}(\cdot)$ denotes the family of spectral projections for
the half-lattice unitary operators $\U_{\pm,k_0}^{(\ga)}$,
\begin{equation}
\U_{\pm,k_0}^{(\ga)}=\oint_{\dD}dE_{\U_{\pm,k_0}^{(\ga)}}(\zeta)\,\zeta.
\end{equation}
Then, the half-lattice Weyl-Titchmarsh m-functions, $m_\pm^{(\ga)}(z,k_0)$, are defined by
\begin{align}
m_\pm^{(\ga)}(z,k_0) &= \pm \De_{k_0}^*(\U_{\pm,k_0}^{(\ga)}+zI)(\U_{\pm,k_0}^{(\ga)}-zI)^{-1}
\De_{k_0} \lb{3.69}
\\
& =\pm
\oint_{\dD}d\Om_{\pm}^{(\ga)}(\zeta,k_0)\,\frac{\zeta+z}{\zeta-z},
\quad z\in\bbC\backslash\dD,\lb{3.70}
\end{align}
with
\begin{equation}
 \oint_{\dD}d\Om_{\pm}^{(\ga)}(\zeta,k_0)= I_m.
\end{equation}
As defined, we note that $m_\pm^{(\ga)}(z,k)$ are matrix-valued Caratheodory functions: see Appendix~\ref{sA} for definition and further properties.

Proof of the next result in the special case when $\ga=I_m$ can be found in \cite[Lemma~2.13, Theorem~2.17]{CGZ07}.
The proof in the case for general unitary $\ga\in\Cm$ follows
the same lines and is omitted here for the sake of brevity.

\begin{theorem} \lb{t3.10}
Let $k_0\in\bbZ$. Then, for each unitary $\ga\in\Cm$, for $P_{\pm,k_0}^{(\ga)}$, $Q_{\pm,k_0}^{(\ga)}$, $R_{\pm,k_0}^{(\ga)}$, $S_{\pm,k_0}^{(\ga)}$ defined in \eqref{3.56} and \eqref{3.57}, and for $m_\pm^{(\ga)}(z,k_0)$ defined in \eqref{3.69}, the following relations hold for $z\in\bbC\backslash(\dD\cup\{0\})$:
\begin{align}
\hatt U^{(\ga)}_\pm(z,\cdot,k_0)=Q_\pm^{(\ga)}(z,\cdot,k_0) + P_\pm^{(\ga)}(z,\cdot,k_0)m_\pm^{(\ga)}(z,k_0) \in
\ltmm{[k_0,\pm\infty)\cap\Z},\lb{3.72}
\\
\hatt V^{(\ga)}_\pm(z,\cdot,k_0)=S_\pm^{(\ga)}(z,\cdot,k_0) + R_\pm^{(\ga)}(z,\cdot,k_0)m_\pm^{(\ga)}(z,k_0) \in
\ltmm{[k_0,\pm\infty)\cap\Z}.\lb{3.73}
\end{align}
Moreover, there exist  unique $\Cm$-valued functions
$M_\pm^{(\ga)}(\cdot,k_0)$ such that for all
$z\in\bbC\backslash(\dD\cup\{0\})$,
\begin{align}
&U_\pm^{(\ga)}(z,\cdot,k_0) = Q_+^{(\ga)}(z,\cdot,k_0) + P_+^{(\ga)}(z,\cdot,k_0)M_\pm^{(\ga)}(z,k_0)
\in \ltmm{[k_0,\pm\infty)\cap\Z},\lb{3.74}
\\
&V_\pm^{(\ga)}(z,\cdot,k_0) = S_+^{(\ga)}(z,\cdot,k_0) + R_+^{(\ga)}(z,\cdot,k_0)M_\pm^{(\ga)}(z,k_0)
\in \ltmm{[k_0,\pm\infty)\cap\Z}.\lb{3.75}
\end{align}
\end{theorem}

\begin{remark}\lb{r3.11}
Within the proof of Theorem~\ref{t3.10}, one observes that the sequences,
\begin{equation}\lb{3.76}
\binom{U_\pm^{(\ga)}(z,k,k_0)}{V_\pm^{(\ga)}(z,k,k_0)}_{k\in\Z},
\end{equation}
defined by \eqref{3.74} and \eqref{3.75},
are unique up to right-multiplication by constant $m\times m$ matrices.
Hence, we shall call $U_\pm^{(\ga)}(z,\cdot,k_0)$ the Weyl--Titchmarsh solutions of $\U$.  Note also that $m_\pm^{(\ga)}(z,k_0)$, as well as $M_\pm^{(\ga)}(z,k_0)$, are said to be half-lattice Weyl--Titchmarsh $m$-functions associated with
$\U_{\pm,k_0}^{(\ga)}$. (See also \cite{Si04a} for a comparison of various
alternative notions of Weyl--Titchmarsh $m$-functions for
$\U_{+,k_0}^{(\ga)}$ with scalar-valued Verblunsky coefficients.)
\end{remark}

For fixed $k_0\in\Z$, $z\in\bbC\backslash(\dD\cup\{0\})$, and unitary $\ga\in\Cm$, using \eqref{3.70} and Theorem~\ref{t3.10}, we obtain,
\begin{equation}\lb{3.77}
M_+^{(\ga)}(z,k_0)=m_+^{(\ga)}(z,k_0),\quad
M_+^{(\ga)}(0,k_0)=I_m.
\end{equation}
In particular, by \eqref{3.77} and the uniqueness up to right multiplication by constant $m\times m$ matrices noted in Remark~\ref{r3.11}, note for \eqref{3.72}--\eqref{3.75} that
\begin{equation}
\hatt U^{(\ga)}_+( z,k,k_0)= U^{(\ga)}_+( z,k,k_0), \quad
\hatt V^{(\ga)}_+( z,k,k_0)= V^{(\ga)}_+( z,k,k_0).
\end{equation}

Following the line of reasoning presented for the special case when $\ga=I_m$ found in \cite{CGZ07}, one uses the equations in Lemma~\ref{l3.9},  together with Theorem~\ref{t3.10}, to obtain the following equations where $C_3,\ D_3$, and $ C_4,\ D_4$ are defined in \eqref{3.62} and  \eqref{3.63} respectively,  with $\ga=\ga_1=\ga_2$.
\begin{align}
M_-^{(\ga)}(z,k_0)&=
\big[D_3+D_4m_-^{(\ga)}(z,k_0-1)\big]\big[C_3+C_4m_-^{(\ga)}(z,k_0-1)\big]^{-1}
\lb{3.79}\\
&=\big[m_-^{(\ga)}(z,k_0)+I_m + z(m_-^{(\ga)}(z,k_0)-I_m)\big] \\
&\hspace{15pt}\times
\big[m_-^{(\ga)}(z,k_0)+I_m - z(m_-^{(\ga)}(z,k_0)-I_m)\big]^{-1},\,
z\in\bbC\backslash(\dD\cup\{0\}),  \no \\
M_-^{(\ga)}(0,k_0)&=
[\ga^{-1/2}\wti\rho^{-1}_{k_0}\ga^{-1/2}+\ga^{1/2}\rho^{-1}_{k_0}\ga^{1/2}]
[\ga^{-1/2}\wti\rho^{-1}_{k_0}\ga^{-1/2}-\ga^{1/2}\rho^{-1}_{k_0}\ga^{1/2}]^{-1},\lb{3.81}\\
m_-^{(\ga)}(z,k_0)&= \big[z(M_-^{(\ga)}(z,k_0) +I_m) - (M_-^{(\ga)}(z,k_0)-I_m)\big]
\lb{3.82} \\
&\hspace{15pt}\times \big[z(M_-^{(\ga)}(z,k_0) +I_m) + (M_-^{(\ga)}(z,k_0)-I_m)\big]^{-1},\, z\in\bbC\backslash(\dD\cup\{0\}).  \no
\end{align}
By their relations to the Caratheodory functions $m_\pm^{(\ga)}(z,k)$ given in \eqref{3.77} and \eqref{3.79}, we see that $M_\pm^{(\ga)}(z,k)$ are also matrix-valued Caratheodory functions.

Next, we introduce the $\Cm$-valued Schur functions $\Phi_\pm^{(\ga)}(\cdot,k)$,
$k\in\bbZ$, by
\begin{align}
\Phi_\pm^{(\ga)}(z,k) = \big[M_\pm^{(\ga)}(z,k)-I_m\big]
\big[M_\pm^{(\ga)}(z,k)+I_m\big]^{-1}, \quad
z\in\C\backslash\dD. \lb{3.83}
\end{align}
Then, by \eqref{3.82}  and \eqref{3.83}, one verifies that
\begin{align}
M_\pm^{(\ga)}(z,k) &= \big[I_m-\Phi_\pm^{(\ga)}(z,k)\big]^{-1}
\big[I_m+\Phi_\pm^{(\ga)}(z,k)\big], \quad
z\in\C\backslash\dD,
\\
m_-^{(\ga)}(z,k) &= \big[zI_m+\Phi_-^{(\ga)}(z,k)\big]^{-1}
\big[zI_m-\Phi_-^{(\ga)}(z,k)\big], \quad
z\in\C\backslash\dD. \lb{3.85}
\end{align}
Moreover, by \eqref{3.56}, \eqref{3.83}, and Theorem~\ref{t3.10},  it follows, as in \cite[Lemma 2.18]{CGZ07}, that for $k\in\Z,\ z\in\C\backslash\dD$,
\begin{equation}\lb{3.86}
\Phi_\pm^{(\ga)}(z,k) =
\begin{cases}
z\ga^{1/2}V_\pm^{(\ga)}(z,k,k_0)U_\pm^{(\ga)}(z,k,k_0)^{-1}\ga^{1/2},& k \text{ odd,}\\
\ga^{1/2}U_\pm^{(\ga)}(z,k,k_0)V_\pm^{(\ga)}(z,k,k_0)^{-1}\ga^{1/2},& k \text{ even,}
\end{cases}
\end{equation}
where $U_\pm^{(\ga)}(z,k,k_0)$ and $V_\pm^{(\ga)}(z,k,k_0)$ are sequences defined in \eqref{3.74} and \eqref{3.75}. Since the Weyl-Titchmarsh solutions defined in \eqref{3.76} are unique up to right multiplication by a constant complex $m\times m$ matirx, \eqref{3.86} implies that $\ga^{-1/2}\Phi_\pm^{(\ga)}(\cdot,k)\ga^{-1/2}$ is $\ga$-independent, and hence for unitary $\ga_1,\ga_2\in\Cm$, and $k\in\Z$, that
\begin{align}
\Phi_\pm^{(\ga_2)}(\cdot,k)&=\ga_2^{1/2}\ga_1^{-1/2}\Phi_\pm^{(\ga_1)}(\cdot,k)\ga_1^{-1/2}\ga_2^{1/2},\\
M_\pm^{(\ga_2)}(\cdot,k) &=\big[(\ga_2^{-1/2}\ga_1^{1/2}+\ga_2^{1/2}\ga_1^{-1/2})M_\pm^{(\ga_1)}(\cdot,k)+(\ga_2^{-1/2}\ga_1^{1/2}-\ga_2^{1/2}\ga_1^{-1/2})\big]\no\\
&\hspace{15pt}\times \big[(\ga_2^{-1/2}\ga_1^{1/2}-\ga_2^{1/2}\ga_1^{-1/2})
M_\pm^{(\ga_1)}(\cdot,k)+(\ga_2^{-1/2}\ga_1^{1/2}+\ga_2^{1/2}\ga_1^{-1/2})\big]^{-1}.
\end{align}

Full and half-lattice resolvent operators lie at the heart of our analysis of the Weyl-Titchmarsh theory for full and half-lattice CMV operators; in particular, as a tool in obtaining our Borg-Marchenko-type uniqueness results in \cite{CGZ07} for CMV operators with matrix-valued coefficients. Hence, we conclude with a discussion of resolvent operators for a general unitary $\ga\in\Cm$.

First, we note the utility of the identities contained in the following lemma. This lemma was proven in \cite[Lemma~3.2]{CGZ07} for matrix-Laurent polynomial solutions of \eqref{3.22} defined by \eqref{3.56} in the special case when $\ga=I_m$. The identities listed below were central to the derivation of the full-lattice resolvent operator in \cite[Lemma 3.3]{CGZ07}.

\begin{lemma}\lb{l3.12}
Let $k,k_0\in\Z$ and $z\in\Cz$.  Then, for a unitary $\ga\in\Cm$, the following identities hold for the matrix-Laurent polynomial solutions of \eqref{3.22} defined in \eqref{3.56} and \eqref{3.57}:
\begin{align}
& P^{(\ga)}_\pm(z,k,k_0)Q^{(\ga)}_\pm(1/\ol{z},k,k_0)^* +
Q^{(\ga)}_\pm(z,k,k_0)P^{(\ga)}_\pm(1/\ol{z},k,k_0)^* = 2(-1)^{k+1}I_m,\lb{3.89}
\\
& R^{(\ga)}_\pm(z,k,k_0)S^{(\ga)}_\pm(1/\ol{z},k,k_0)^* +
S^{(\ga)}_\pm(z,k,k_0)R^{(\ga)}_\pm(1/\ol{z},k,k_0)^* = 2(-1)^{k}I_m,
\\
& P^{(\ga)}_\pm(z,k,k_0)S^{(\ga)}_\pm(1/\ol{z},k,k_0)^* +
Q^{(\ga)}_\pm(z,k,k_0)R^{(\ga)}_\pm(1/\ol{z},k,k_0)^* = 0,
\\
& R^{(\ga)}_\pm(z,k,k_0)Q^{(\ga)}_\pm(1/\ol{z},k,k_0)^* +
S^{(\ga)}_\pm(z,k,k_0)P^{(\ga)}_\pm(1/\ol{z},k,k_0)^* = 0.\lb{3.92}
\end{align}
\end{lemma}
\begin{proof}
For the case $k=k_0$, each of the equations \eqref{3.89}--\eqref{3.92} follows from \eqref{3.56} and \eqref{3.57}. The induction argument described in \cite[Lemma~3.2]{CGZ07} then suffices to establish \eqref{3.89}--\eqref{3.92} when $k\ne k_0$. As already noted, the proof in \cite[Lemma~3.2]{CGZ07} treats the special case when $\ga=I_m$ for solutions defined by \eqref{3.56}. In all cases under consideration, the proof involves a number of cases all following a similar pattern. We outline one of these cases for a solution of \eqref{3.22} defined in \eqref{3.57}.

Suppose equations \eqref{3.89}--\eqref{3.92} hold when $k\in\Z$ is even.
Then utilizing \eqref{3.22} together with \eqref{3.8} and \eqref{3.9},
one computes
\begin{align}
&P^{(\ga)}_-(z,k+1,k_0)Q^{(\ga)}_-(1/\ol{z},k+1,k_0)^* +
Q^{(\ga)}_-(z,k+1,k_0)P^{(\ga)}_-(1/\ol{z},k+1,k_0)^* \no
\\
&\hspace{5pt} = \wti\rho_{k+1}^{-1}\al_{k+1}
\big[P^{(\ga)}_-(z,k,k_0)Q^{(\ga)}_-(1/\ol{z},k,k_0)^* +
Q^{(\ga)}_-(z,k,k_0)P^{(\ga)}_-(1/\ol{z},k,k_0)^*\big] \al_{k+1}^*\wti\rho_{k+1}^{-1}
\no
\\
&\hspace{15pt} + \wti\rho_{k+1}^{-1} \big[R^{(\ga)}_-(z,k,k_0)S^{(\ga)}_-(1/\ol{z},k,k_0)^*
+ S^{(\ga)}_-(z,k,k_0)R^{(\ga)}_-(1/\ol{z},k,k_0)^*\big] \wti\rho_{k+1}^{-1}\no
\\
&\hspace{15pt} + z\wti\rho_{k+1}^{-1} \big[R^{(\ga)}_-(z,k,k_0)Q^{(\ga)}_-(1/\ol{z},k,k_0)^*
+ S^{(\ga)}_-(z,k,k_0)P^{(\ga)}_-(1/\ol{z},k,k_0)^*\big]
\al_{k_0}^*\wti\rho_{k+1}^{-1} \no
\\
&\hspace{15pt} + \wti\rho_{k+1}^{-1}\al_{k_0}
\big[P^{(\ga)}_-(z,k,k_0)S^{(\ga)}_-(1/\ol{z},k,k_0)^* +
Q^{(\ga)}_-(z,k,k_0)R^{(\ga)}_-(1/\ol{z},k,k_0)^*\big] \wti\rho_{k+1}^{-1}z^{-1} \no
\\
&\hspace{5pt} = 2(-1)^{k+1}
\big[\wti\rho_{k+1}^{-1}\al_{k+1}\al_{k+1}^*\wti\rho_{k+1}^{-1} -
\wti\rho_{k+1}^{-2}\big] = 2(-1)^{(k+1)+1}I_m.
\end{align}
Similarly, one checks all remaining cases at the point $k+1$. Then by inverting the matrix $\T(z,k)$ and utilizing \eqref{3.22} in the form
\begin{align}
\binom{P_-(z,k-1,k_0)}{R_-(z,k-1,k_0)}  = \T(z,k)^{-1}
\binom{P_-(z,k,k_0)}{R_-(z,k,k_0)},
\end{align}
where
\begin{equation}
\T(z,k)^{-1}=
\begin{cases}
\begin{pmatrix}
-\rho^{-1}_k\al^*_k& z\rho^{-1}_k\\z^{-1}\wti\rho^{-1}_k&-\wti\rho^{-1}_k\al_k
\end{pmatrix}
&k\ odd,\\[15pt]
\begin{pmatrix}
-\wti\rho^{-1}_k\al_k&\wti\rho^{-1}_k\\ \rho^{-1}_k&-\rho^{-1}_k\al^*_k
\end{pmatrix}
&k\ even,
\end{cases}
\end{equation}
one verifies the equations \eqref{3.89}--\eqref{3.92} at the point $k-1$. Similarly,
one verifies \eqref{3.89}--\eqref{3.92} at the points $k+1$ and $k-1$
under the assumption that $k\in\Z$ is odd.
\end{proof}

The next lemma introduces the half-lattice resolvent operators for $\U_{\pm,k_0}^{(\ga)}$; this appears to be a new result:

\begin{lemma}\lb{l3.13}
Let $z\in\bbC\backslash(\dD\cup\{0\})$ and fix $k_0\in\bbZ$. Then, for a unitary $\ga\in\Cm$, the
resolvent $(\U_{\pm,k_0}^{(\ga)}-zI)^{-1}$ for the unitary CMV operator $\U_{\pm,k_0}^{(\ga)}$  on
$\ltm{[k_0,\pm\infty)\cap\bbZ}$ is given in terms of its matrix representation in the
standard basis of $\ltm{[k_0,\pm\infty)\cap\bbZ}$ by
\begin{align}
&(\U_{+,k_0}^{(\ga)}-zI)^{-1}=\frac1{2z}
\begin{cases}
-P^{(\ga)}_+(z,k,k_0)\hatt U^{(\ga)}_+(1/\bar z,k',k_0)^*,\\
\hspace{48pt} k<k'\ and\ k=k'\ odd,\\
\hatt U^{(\ga)}_+( z,k,k_0)P^{(\ga)}_+(1/\bar z,k',k_0)^*,\\
\hspace{33pt} k>k'\ and\ k=k'\ even,
\end{cases}
k, k'\in[k_0,\infty)\cap\bbZ,\lb{3.95}
\intertext{where $P^{(\ga)}_+(z,k,k_0)$ is defined in \eqref{3.56}, and
$\hatt U^{(\ga)}_+( z,k,k_0)$ is defined in \eqref{3.72};}
&(\U_{-,k_0}^{(\ga)}-zI)^{-1}=\frac1{2z}
\begin{cases}
-\hatt U^{(\ga)}_-(z,k,k_0)P^{(\ga)}_-(1/\bar z,k',k_0)^*,\\
\hspace{48pt} k<k'\ and\ k=k'\ odd,\\
P^{(\ga)}_-(z,k,k_0)\hatt U^{(\ga)}_-(1/\bar z,k',k_0)^*,\\
\hspace{33pt} k>k'\ and\ k=k'\ even,
\end{cases}
\hspace{-5pt}k, k'\in(-\infty,k_0]\cap\bbZ,\lb{3.96}
\end{align}
where  $P^{(\ga)}_-(z,k,k_0)$ is defined in \eqref{3.57}, and $\hatt U^{(\ga)}_-( z,k,k_0)$ is defined in \eqref{3.72}.
\end{lemma}
\begin{proof}
We begin by noting that the following equations hold for $k\in\bbZ$:
\begin{align}
&R^{(\ga)}_\pm(z,k,k_0)\hatt U^{(\ga)}_\pm(1/\bar z,k,k_0)^* +
\hatt V^{(\ga)}_\pm( z,k,k_0)P^{(\ga)}_\pm(1/\bar z,k,k_0)^*=0,\lb{3.98}\\
&P^{(\ga)}_\pm(z,k,k_0)\hatt U^{(\ga)}_\pm( 1/\bar z,k,k_0)^* +
\hatt U^{(\ga)}_\pm( z,k,k_0)P^{(\ga)}_\pm(1/\bar z,k,k_0)^* = 2(-1)^{k+1}I_m.\lb{3.99}
\end{align}
These equations are a consequence of \eqref{3.72}, \eqref{3.73}, \eqref{3.89}, \eqref{3.92}, and the fact that
$m^{(\ga)}_\pm(z,k_0)$ are matrix-valued Caratheodory functions and hence satisfy the property given in \eqref{A.7}.

For $k,k'\in[k_0,\infty)\cap\bbZ$, let $G^{(\ga)}_+(z,k,k',k_0)$ be defined  by
\begin{equation}
G^{(\ga)}_+(z,k,k',k_0)=
\begin{cases}
-P^{(\ga)}_+(z,k,k_0)\hatt U^{(\ga)}_+(1/\bar z,k',k_0)^*,& k<k'\ and\ k=k'\ odd,\\
\hatt U^{(\ga)}_+( z,k,k_0)P^{(\ga)}_+(1/\bar z,k',k_0)^*,& k>k'\ and\ k=k'\ even.
\end{cases}
\end{equation}
Then, \eqref{3.95} is equivalent to showing that
\begin{equation}\lb{3.101}
\big(\U_{+,k_0}^{(\ga)}-zI\big)G^{(\ga)}_+(z,\cdot,k',k_0) = 2z\De_{k'},\quad k'\in[k_0,\infty)\cap\bbZ.
\end{equation}

Assume that $k'\in[k_0,\infty)\cap\bbZ$ is odd. Then, for $\ell \in([k_0,\infty)\cap\bbZ)\backslash
\{k',k'+1\}$, note that
\begin{align}\lb{3.102}
\big(\big(\U_{+,k_0}^{(\ga)}-zI\big)G^{(\ga)}_+(z,\cdot,k',k_0)\big)(\ell)
= \big(\big(\V_{+,k_0}^{(\ga)}\W_{+,k_0}^{(\ga)}-zI\big)G^{(\ga)}_+(z,\cdot,k',k_0)\big)(\ell) =0.
\end{align}
Next, by \eqref{3.98}, \eqref{3.99}, note that
\begin{align}
&\begin{pmatrix}
\big(\big(\U_{+,k_0}^{(\ga)}-zI\big)G^{(\ga)}_+(z,\cdot,k',k_0)\big)(k')\\
\big(\big(\U_{+,k_0}^{(\ga)}-zI\big)G^{(\ga)}_+(z,\cdot,k',k_0)\big)(k'+1)
\end{pmatrix}\no
\\
&\hspace{10pt}=
\begin{pmatrix}
\big(\big(\V_{+,k_0}^{(\ga)}\W_{+,k_0}^{(\ga)}-zI\big)G^{(\ga)}_+(z,\cdot,k',k_0)\big)(k')\\
\big(\big(\V_{+,k_0}^{(\ga)}\W_{+,k_0}^{(\ga)}-zI\big)G^{(\ga)}_+(z,\cdot,k',k_0)\big)(k'+1)
\end{pmatrix}\no
\\
&\hspace{10pt}=\Te(k'+1)
\begin{pmatrix}
-zR^{(\ga)}_+(z,k',k_0)\hatt U^{(\ga)}_+(1/\bar z,k',k_0)^*\\
z
\hatt V^{(\ga)}_+( z,k'+1,k_0)P^{(\ga)}_+(1/\bar z,k',k_0)^*
\end{pmatrix}
-z\begin{pmatrix}
G^{(\ga)}_+(z,k',k',k_0))\\G^{(\ga)}_+(z,k'+1,k',k_0))
\end{pmatrix}\no
\\
&\hspace{10pt}=\Te(k'+1)
\begin{pmatrix}
z\hatt V^{(\ga)}_+( z,k',k_0)P^{(\ga)}_+(1/\bar z,k',k_0)^*\\
z
\hatt V^{(\ga)}_+( z,k'+1,k_0)P^{(\ga)}_+(1/\bar z,k',k_0)^*
\end{pmatrix}
-z\begin{pmatrix}
G^{(\ga)}_+(z,k',k',k_0))\\G^{(\ga)}_+(z,k'+1,k',k_0))
\end{pmatrix}\no
\\
&\hspace{10pt}=
z\begin{pmatrix}
\hatt U^{(\ga)}_+( z,k,k_0)P^{(\ga)}_+(1/\bar z,k,k_0)^* +
P^{(\ga)}_+(z,k,k_0)\hatt U^{(\ga)}_+( 1/\bar z,k,k_0)^*
\\
0
\end{pmatrix}\no
\\
&\hspace{10pt}=
\begin{pmatrix}
2z(-1)^{k'+1}I_m\\0
\end{pmatrix}
=\begin{pmatrix}2zI_m\\0 \end{pmatrix}.\lb{3.103}
\end{align}
Hence, when $k'\in[k_0,\infty)\cap\bbZ$ is odd, \eqref{3.101} is a consequence of \eqref{3.102} and \eqref{3.103}.

Assume that $k'\in[k_0,\infty)\cap\bbZ$ is even. Then, for $\ell\in ([k_0,\infty)\cap\bbZ)\backslash
\{k'-1,k'\}$, note that \eqref{3.102} holds. Again, by \eqref{3.98} and \eqref{3.99}, note that
\begin{align}
&\begin{pmatrix}
\big(\big(\U_{+,k_0}^{(\ga)}-zI\big)G^{(\ga)}_+(z,\cdot,k',k_0)\big)(k'-1)\\
\big(\big(\U_{+,k_0}^{(\ga)}-zI\big)G^{(\ga)}_+(z,\cdot,k',k_0)\big)(k')
\end{pmatrix}\no
\\
&\hspace{7pt}=
\begin{pmatrix}
\big(\big(\V_{+,k_0}^{(\ga)}\W_{+,k_0}^{(\ga)}-zI\big)G^{(\ga)}_+(z,\cdot,k',k_0)\big)(k'-1)\\
\big(\big(\V_{+,k_0}^{(\ga)}\W_{+,k_0}^{(\ga)}-zI\big)G^{(\ga)}_+(z,\cdot,k',k_0)\big)(k')
\end{pmatrix}\no
\\
&\hspace{7pt}=\Te(k')
\begin{pmatrix}
-zR^{(\ga)}_+(z,k'-1,k_0)\hatt U^{(\ga)}_+(1/\bar z,k'-1,k_0)^*\\
z
\hatt V^{(\ga)}_+( z,k',k_0)P^{(\ga)}_+(1/\bar z,k',k_0)^*
\end{pmatrix}
-z\begin{pmatrix}
G^{(\ga)}_+(z,k'-1,k',k_0))\\G^{(\ga)}_+(z,k',k',k_0))
\end{pmatrix}\no
\\
&\hspace{7pt}=\Te(k')
\begin{pmatrix}
-zR^{(\ga)}_+(z,k'-1,k_0)\hatt U^{(\ga)}_+(1/\bar z,k'-1,k_0)^*\\
-zR^{(\ga)}_+(z,k',k_0)\hatt U^{(\ga)}_+(1/\bar z,k',k_0)^*
\end{pmatrix}
-z\begin{pmatrix}
G^{(\ga)}_+(z,k'-1,k',k_0))\\G^{(\ga)}_+(z,k',k',k_0))
\end{pmatrix}\no
\\
&\hspace{7pt}=
-z\begin{pmatrix}
0\\
P^{(\ga)}_+(z,k,k_0)\hatt U^{(\ga)}_+( 1/\bar z,k,k_0)^*+\hatt U^{(\ga)}_+( z,k,k_0)P^{(\ga)}_+(1/\bar z,k,k_0)^*
\end{pmatrix}\no
\\
&\hspace{7pt}=
\begin{pmatrix}
0\\
2z(-1)^{k'+2}I_m
\end{pmatrix}
=\begin{pmatrix}0\\2zI_m \end{pmatrix}.\lb{3.104}
\end{align}
Hence, when $k'\in[k_0,\infty)\cap\bbZ$ is even, \eqref{3.101} is a consequence of \eqref{3.102} and \eqref{3.104}.

The proof of \eqref{3.96} is omitted here for brevity, but follows a line of reasoning similar that just completed for the proof of \eqref{3.95} by using \eqref{3.72}, \eqref{3.73}, \eqref{3.98}, and \eqref{3.99}.
\end{proof}

Before stating our final result for the full-lattice resolvent of $\U$, let us recall the definition of, and some facts about, the matrix-valued Wronskian, defined in \cite{CGZ07}, for two $\Cm$-valued sequences $U_j(z,\cdot)$, $j=1,2$. First, the Wronskian is defined for $k\in\bbZ, \; z\in\Cz,$ by
\begin{align}\lb{3.105}
&W(U_1(1/\ol{z},k),U_2(z,k))\no\\
&\hspace{10pt}= \frac{(-1)^{k+1}}{2} \big[
U_1(1/\ol{z},k)^*U_2(z,k)-(\V^*U_1(1/\ol{z},\dott))(k)^*
(\V^*U_2(z,\dott))(k)\big],
\end{align}
where $\V$ is defined in \eqref{3.15}. It is shown in \cite[Lemma 3.1]{CGZ07} when $\U U_j(z,\cdot)=zU_j(z,\cdot)$, and hence $\V^* U_j(z,\cdot)= V_j(z,\cdot)$, $j=1,2$, where $\U$ is viewed as a difference expression rather than as an operator acting on $\ell^2(\bbZ)^{m\times m}$, that the Wronskian of $U_j(z,\cdot)$, $j=1,2$, is independent of $k\in\Z$. Moreover, for $P_+^{(\ga)}(z,\cdot,k_0)$ and $Q_+^{(\ga)}(z,\cdot,k_0)$ defined in \eqref{3.56}, and for $U_\pm^{(\ga)}(z,\cdot,k_0)$ defined in \eqref{3.74}, with $k,k_0\in\bbZ, \; z\in\Cz$, as a consequence of \eqref{3.56}, \eqref{3.57}, \eqref{3.74}, \eqref{3.75}, and property \eqref{A.7}, we see that
\begin{align}
W\big(P_+^{(\ga)}(1/\ol{z},k,k_0),Q_+^{(\ga)}(z,k,k_0)\big) &= I_m,
\\
W\big(U_+^{(\ga)}(1/\ol{z},k,k_0),U_-^{(\ga)}(z,k,k_0)\big) &= M_-^{(\ga)}(z,k_0)-M_+^{(\ga)}(z,k_0).\lb{3.107}
\end{align}

For notational simplicity, we abbreviate the Wronskian of $U_+^{(\ga)}$ and
$U_-^{(\ga)}$ by
\begin{equation}
W^{(\ga)}(z,k_0)= - W\big(U_+^{(\ga)}(1/\ol{z},k,k_0),U_-^{(\ga)}(z,k,k_0)\big).
\end{equation}
Then, using \eqref{3.77}, \eqref{3.81}, and \eqref{3.107},  one analytically continues $W^{(\ga)}(z,k_0)$ to $z=0$ and obtains
\begin{equation}\lb{3.109}
W^{(\ga)}(z,k_0) = M^{(\ga)}_+(z,k_0)-M^{(\ga)}_-(z,k_0), \quad k\in\bbZ, \; z\in\bbC.
\end{equation}
Moreover, one verifies the following symmetry property of the Wronskian $W^{(\ga)}(z,k_0)$, for $k\in\Z, \; z\in\C$
\begin{equation}
M^{(\ga)}_+(z,k_0)W^{(\ga)}(z,k_0)^{-1}M^{(\ga)}_-(z,k_0) =
M^{(\ga)}_-(z,k_0)W^{(\ga)}(z,k_0)^{-1}M^{(\ga)}_+(z,k_0).
\end{equation}
Then, using \eqref{3.74}, \eqref{3.75}, \eqref{3.89}, \eqref{3.92}, \eqref{3.107}, \eqref{3.109}, and following the steps in the proof for \cite[Lemma 3.2]{CGZ07} for the special case when $\ga=I_m$, we find that
\begin{align}
&U^{(\ga)}_+(z,k,k_0){W^{(\ga)}(z,k_0)}^{-1}U^{(\ga)}_-(1/\ol{z},k,k_0)^*\no\\
&\quad - U^{(\ga)}_-(z,k,k_0)W^{(\ga)}(z,k_0)^{-1}U^{(\ga)}_+(1/\ol{z},k,k_0)^* = 2(-1)^{k+1}I_m,
\lb{3.111}
\\
&V^{(\ga)}_+(z,k,k_0)W^{(\ga)}(z,k_0)^{-1}U^{(\ga)}_-(1/\ol{z},k,k_0)^*\no\\
&\quad -
V^{(\ga)}_-(z,k,k_0)W^{(\ga)}(z,k_0)^{-1}U^{(\ga)}_+(1/\ol{z},k,k_0)^* = 0. \lb{3.112}
\end{align}
Then, using \eqref{3.111} and \eqref{3.112}, and following the steps in the proof for \cite[Lemma 3.3]{CGZ07} for the special case when $\ga=I_m$, we obtain the next result for the resolvent of the full-lattice operator $\U$.

\begin{lemma} \lb{l3.14}
Let $z\in\bbC\backslash(\dD\cup\{0\})$, fix $k_0\in\bbZ$, and let $\ga\in\Cm$ be unitary. Then the
resolvent $(\U-zI)^{-1}$ of the unitary CMV operator $\U$  on
$\ltm{\bbZ}$ is given, for $k,k' \in\Z$, in terms of its matrix representation in the
standard basis of $\ltm{\bbZ}$ by
\begin{align}
(\U-zI)^{-1}(k,k') = \frac{1}{2z}
\begin{cases}
U^{(\ga)}_-(z,k,k_0)W^{(\ga)}(z,k_0)^{-1}U^{(\ga)}_+(1/\ol{z},k',k_0)^*,\\
\hspace*{4.1cm} k < k' \text{ or } k = k' \text{ odd},
\\
U^{(\ga)}_+(z,k,k_0)W^{(\ga)}(z,k_0)^{-1}U^{(\ga)}_-(1/\ol{z},k',k_0)^*,\\
\hspace*{3.92cm} k > k' \text{ or } k = k' \text{ even},
\end{cases} \lb{3.113}
\end{align}
Moreover, since $0\in\bbC\backslash\sigma(\U)$, \eqref{3.113}
analytically extends to $z=0$.
\end{lemma}

\appendix
\section{Basic Facts on Caratheodory and Schur Functions}
\lb{sA}
\renewcommand{\theequation}{A.\arabic{equation}}
\renewcommand{\thetheorem}{A.\arabic{theorem}}
\setcounter{theorem}{0} \setcounter{equation}{0}

In this appendix we summarize a few basic facts on matrix-valued
Caratheodory and Schur functions used throughout this manuscript.
(For the analogous case of matrix-valued Herglotz functions we refer
to \cite{GT00} and the extensive list of references therein.)

We denote by $\D$ and $\dD$ the open unit disk and the
counterclockwise oriented unit circle in the complex plane $\C$,
\begin{equation}
\D = \{ z\in\C \st \abs{z} < 1 \}, \quad \dD = \{ \ze\in\C \st
\abs{\ze} = 1 \}.
\end{equation}
Moreover, we denote as usual $\Re(A)=(A+A^*)/2$ and
$\Im(A)=(A-A^*)/(2i)$ for square matrices $A$ with complex-valued
entries.

\begin{definition} \lb{dA.1}
Let $m\in\bbN$ and $F_\pm$, $\Phi_+$, and $\Phi_-^{-1}$ be $m\times
m$
matrix-valued analytic functions in $\D$. \\
$(i)$ $F_+$ is called a {\it Caratheodory matrix} if $\Re(F_+(z))\geq
0$ for all $z\in\D$ and $F_-$ is called an {\it anti-Caratheodory
matrix} if $-F_-$ is a
Caratheodory matrix. \\
$(ii)$ $\Phi_+$ is called a {\it Schur matrix} if
$\|\Phi_+(z)\|_{\Cm} \leq 1$, for all $z\in\D$.\  $\Phi_-$ is called
an {\it anti-Schur matrix} if $\Phi_-^{-1}$ is a Schur matrix.
\end{definition}

\begin{theorem} \lb{tA.2}
Let $F$ be an $m\times m$ Caratheodory matrix, $m\in\bbN$. Then $F$
admits the Herglotz representation
\begin{align}
& F(z)=iC+ \oint_{\dD} d\Omega(\zeta) \, \f{\zeta+z}{\zeta-z}, \quad
z\in\D, \lb{A.3}
\\
& C=\Im(F(0)), \quad \oint_{\dD} d\Omega(\zeta) = \Re(F(0)),
\end{align}
where $d\Omega$ denotes a nonnegative $m \times m$ matrix-valued
measure on $\dD$. The measure $d\Omega$ can be reconstructed from $F$
by the formula
\begin{equation}
\Omega\big(\Arc\big(\big(e^{i\te_1},e^{i\te_2}\big]\big)\big)
=\lim_{\delta\downarrow 0} \lim_{r\uparrow 1} \f{1}{2\pi}
\oint_{\te_1+\delta}^{\te_2+\delta} d\te \,
\Re\big(F\big(r\zeta\big)\big), \lb{A.4}
\end{equation}
where
\begin{equation}
\Arc\big(\big(e^{i\theta_1},e^{i\theta_2}\big]\big)
=\big\{\zeta\in\dD\,|\, \theta_1<\te\leq \theta_2\big\}, \quad
\theta_1 \in [0,2\pi), \; \theta_1<\theta_2\leq \theta_1+2\pi.
\lb{A.5}
\end{equation}
Conversely, the right-hand side of equation \eqref{A.3} with $C =
C^*$ and $d\Omega$ a finite nonnegative $m \times m$ matrix-valued
measure on $\dD$ defines a Caratheodory matrix.
\end{theorem}

We note that additive nonnegative $m\times m$ matrices on the
right-hand side of \eqref{A.3} can be absorbed into the measure
$d\Om$ since
\begin{equation}
\oint_\dD d\mu_0(\zeta) \, \f{\zeta+z}{\zeta-z}=1, \quad z\in\D,
\lb{A.5a}
\end{equation}
where
\begin{equation}
d\mu_0(\zeta)=\f{d\te}{2\pi}, \quad \zeta=e^{i\te}, \; \te\in
[0,2\pi) \lb{A.5b}
\end{equation}
denotes the normalized Lebesgue measure on the unit circle $\dD$.

Given a Caratheodory (resp., anti-Caratheodory) matrix $F_+$ (resp.
$F_-$) defined in $\D$ as in \eqref{A.3}, one extends $F_\pm$ to all
of $\bbC\backslash\dD$ by
\begin{equation}
F_\pm(z)=iC_\pm \pm \oint_{\dD} d\Om_\pm (\zeta) \,
\f{\zeta+z}{\zeta-z}, \quad z\in\bbC\backslash\dD, \;\;
C_\pm=C_\pm^*. \lb{A.6}
\end{equation}
In particular,
\begin{equation}
F_\pm(z) = -F_\pm(1/\ol{z})^*, \quad z\in\C\backslash\ol{\D}.
\lb{A.7}
\end{equation}
Of course, this continuation of $F_\pm|_{\D}$ to
$\bbC\backslash\ol\D$, in general, is not an analytic continuation of
$F_\pm|_\D$.

Next, given the functions $F_\pm$ defined in $\bbC\backslash\dD$ as
in \eqref{A.6}, we introduce the functions $\Phi_\pm$ by
\begin{equation}
\Phi_\pm(z)=[F_\pm(z)-I_m][F_\pm(z)+I_m]^{-1}, \quad
z\in\bbC\backslash\dD.  \lb{A.11}
\end{equation}
We recall (cf., e.g., \cite[p.\ 167]{SF70}) that if $\pm \Re(F_\pm)
\geq 0$, then $[F_\pm \pm I_m]$ is invertible. In particular,
$\Phi_+|_{\D}$ and $[\Phi_-]^{-1}|_{\D}$ are Schur matrices (resp.,
$\Phi_-|_{\D}$ is an anti-Schur matrix). Moreover,
\begin{equation}
F_\pm(z)= [I_m-\Phi_\pm (z)]^{-1} [I_m+\Phi_\pm (z)], \quad
z\in\bbC\backslash\dD.    \lb{A.12}
\end{equation}

\noindent {\bf Acknowledgments.}
Fritz Gesztesy would like to thank all organizers of the 14th International Conference on Difference Equations and Applications (ICDEA 2008), for their kind invitation and the stimulating atmosphere created during the meeting. In addition, he is particularly indebted to Mehmet \"Unal for the extraordinary hospitality extended to him during his ten day stay in Istanbul in July of 2008.



\end{document}